\documentclass[10pt]{amsart}
\usepackage[small]{caption}
\usepackage{amsmath,amsfonts,amscd,amssymb,epsf}
\newtheorem{theorem}{Theorem}[section]
\newtheorem{proposition}[theorem]{Proposition}
\newtheorem{lemma}[theorem]{Lemma}

\theoremstyle{definition}
\newtheorem{definition}[theorem]{Definition}

\theoremstyle{remark} 
\numberwithin{equation}{section}

\newcommand{\Z}{{\mathbb{Z}}}

\newcommand{\C}{\mathbb{C}}

\newcommand{\R}{\mathbb{R}}

\newcommand{\pa}{\partial}

\newcommand{\vep}{\varepsilon}

\begin{document}

\title[Gaussian processes with Cauchy covariance]{ Gaussian fields and Gaussian sheets with generalized Cauchy covariance structure}
\author{S.C. Lim$^1$}\email{$^1$sclim@mmu.edu.my}\author{L.P.
Teo$^{2}$}\email{$^2$lpteo@mmu.edu.my}

 \maketitle

\noindent {\scriptsize \hspace{1cm}$^1$Faculty of Engineering,
Multimedia University, Jalan Multimedia, }

\noindent {\scriptsize \hspace{1.1cm} Cyberjaya, 63100, Selangor
Darul Ehsan, Malaysia.}

\noindent {\scriptsize \hspace{1cm} $^2$Faculty of Information
Technology, Multimedia University, Jalan Multimedia,}

\noindent{\scriptsize \hspace{1.1cm} Cyberjaya, 63100, Selangor
Darul Ehsan, Malaysia.}
\begin{abstract}
Two types of Gaussian processes, namely the Gaussian field with
generalized Cauchy covariance (GFGCC) and the Gaussian sheet with
generalized Cauchy covariance (GSGCC) are considered. Some of the
basic properties and the asymptotic properties of the spectral
densities of these random fields are studied. The associated
self-similar random fields obtained by applying the Lamperti
transformation to GFGCC and GSGCC are studied.

\end{abstract}

\section{Introduction}  Generalization of some well-known stochastic
processes indexed by a single parameter to processes indexed by two
parameters has attracted considerable interest recently. For
example, Houdre and Villa \cite{Houdre_Villa} generalized fractional
Brownian motion parametrized by a single Hurst index to the
bifractional
  Brownian motion characterized by two indices. Another
example is provided by the  multidimensional stationary Gaussian
fields with generalized Cauchy covariance indexed by two parameters
introduced by Gneiting and Schlather \cite{Gneiting_Schlather}.
These processes can be regarded as extension of the Gaussian
processes with Cauchy covariance used in geostatistics. For
simplicity, we call such processes the Gaussian field with
generalized Cauchy covariance (GFGCC). Here we would like to point
out that one should not confuse such a process with a stable process
with Cauchy marginals.

In general, processes parametrized by two indices can provide more
flexibility in their applications in modeling physical phenomena. In
particular, the GFGCC model has an additional nice and useful
property as it allows separate characterization of fractal dimension
and long range dependence (LRD) by two different indices. This is in
contrast to models based on fractional Brownian motion (or
fractional Brownian noise) which use a single index to characterize
these two properties. Models based on a stochastic process or field
parametrized by a single index seem inadequate. A detailed analysis
on network traffic carried by Park et al \cite{Park} shows that
fractional Brownian motion/fractional Gaussian noise model is
inadequate for description of network traffic for all scales since
at very small time scales the traffic fluctuations are no longer
statistically self-similar. The need to replace global scaling by
local scaling is essential for processes such as multifractional
Brownian motion introduced independently by Peltier and Vehel
\cite{Peltier_Vehel} and Benassi, Jaffard and Roux
\cite{Benassi_Jaffard_Roux}.

     The main aim of this paper is to study  GFGCC and its anisotropic counterpart, which we call
     Gaussian sheet with generalized Cauchy covariance (GSGCC). In view of the fact
that   GFGCC is widely used in geostatistics and other applications
\cite{b,Berizzi_Mese_Martorella, e,f,g,h,i,j}, it will be useful to
consider its   properties in more detail. In these existing
applications, usually only the covariance structure of GFGCC are
used, and the sample path properties of GFGCC are rarely mentioned.
However, a better understanding of the sample properties of GFGCC
and GSGCC will render more versatility and flexibility to their
applications. Our approach to this subject is mainly from a physical
viewpoint. Basic sample properties such as the long range dependence
and the local self-similarity properties of GFGCC and GSGCC are
investigated. A simpler method is used to derive the asymptotic
properties of the spectral densities of GFGCC and GSGCC. By
generalizing the Lamperti transformation to $n$-dimensional
processes,   new types of random field and random sheet with global
self-similar property associated with GFGCC and GSGCC are obtained.
Properties of these random field and random sheet are also studied.

\section{ Isotropic Gaussian Field with Generalized Cauchy Covariance}
In this section we consider GFGCC, which is a multidimensional
isotropic Gaussian random field in $n$-dimensional Euclidean space.
We first introduce some notations and state some basic definitions
and properties of GFGCC. Denote by $\mathbb{N}$, $\Z$, $\R$ and
$\R_+$ the sets of positive integers, integers, real numbers and
positive real numbers respectively. Let $t=(t_1,\ldots, t_n)$ and
$s=(s_1,\ldots, s_n)$ be two vectors in $\R^n$, and $\Vert
t\Vert=\sqrt{\sum_{i=1}^n t_i^2}$, $n, j \in\mathbb{N}$ be its
Euclidean norm. By $t\rightarrow 0^+$ and $t\rightarrow \infty$, we
mean $t_i\rightarrow 0^+$ and
 $t_i\rightarrow \infty$ respectively for all $i=1,\ldots, n$.

\begin{definition}    A random field $X_{\alpha,\beta}(t)$ on $\R^n$
 is called a Gaussian field with generalized Cauchy covariance (or GFGCC) if it is a stationary Gaussian field with mean zero and
covariance given by \begin{align}\label{eq1} C_{\alpha,\beta}(
\tau)=\langle X_{\alpha,\beta}(t+ \tau)X_{\alpha,\beta}(t)\rangle
=\left(1+\Vert \tau\Vert^{\alpha}\right)^{-\beta},
\end{align} where $\alpha\in (0,2]$  and $\beta>0$.\end{definition}
Note that \eqref{eq1} has the same functional form as the
characteristic function of the generalized multivariate Linnik
distribution first studied by Anderson \cite{Anderson}.
$C_{\alpha,\beta}(\tau)$ is positive-definite for the above ranges
of $\alpha$ and $\beta$, and it is  completely monotone for $0
<\alpha\leq  1$, $\beta> 0$. $X_{\alpha,\beta}(t)$ becomes the
Gaussian field with usual Cauchy covariance
 when $\alpha=2, \beta=1$.

Recall that a random field $X(t)$ is $H$-self-similar ($H$ss) if
$X(ct)=_d c^HX(t)$, where $=_d$ denotes equality in the sense of
finite-dimensional distributions of $X$. Self-similar property
requires scale invariance to hold for all scales. This is rather too
restrictive for many applications. We also know from Samorodnitsky
and Taqqu \cite{Samorodnitsky_Taqqu} that a stationary Gaussian
random field such as $X_{\alpha,\beta}(t)$ can not be a self-similar
field. However, $X_{\alpha,\beta}(t)$ satisfies a weaker
self-similar property known as local self-similarity considered by
Kent and Wood \cite{Kent_Wood}.

\begin{definition}\label{def2}Let $\alpha\in (0, 2]$. A centered stationary Gaussian field is locally
self-similar (lss) of order $\alpha/2$ if for
$\Vert\tau\Vert\rightarrow 0^+$, its covariance $C(\tau)$ satisfies
\begin{align}\label{eq2}
C(\tau)=A-B\Vert\tau\Vert^{\alpha}\left[1+O\left(\Vert
\tau\Vert^{\delta}\right)\right]
\end{align}for some positive constants $A$, $B$ and $\delta$.
\end{definition}

Since as $\Vert\tau\Vert\rightarrow 0^+$,
\begin{align}\label{eq2_2}C_{\alpha,\beta}(\tau)=1-\beta
\Vert\tau\Vert^{\alpha}\left[1+\Vert\tau\Vert^{\alpha}\right],\end{align}
 the GFGCC  $X_{\alpha,\beta}(t)$  is $\alpha/2$ -- lss, with
 $A=1$, $B=\beta$ and
 $\delta=\alpha$. Adler \cite{Adler} called the class of Gaussian fields which satisfy
\eqref{eq2} the indexed-$\alpha$ fields. These processes are also
known as the Adler processes according to some authors, for example
Lang and Roueff \cite{Lang_Roueff}. They form a very rich class of
Gaussian random fields, which include the centered Gaussian field
$\Xi_{\alpha,\beta}(t)$ with powered exponential covariance
\begin{align}\label{eq22_1} \left\langle
\Xi_{\alpha,\beta}(t+\tau)\Xi_{\alpha,\beta}(t)\right\rangle
=e^{-\beta\Vert\tau\Vert^{\alpha}},\end{align}which have the same
functional form as the characteristic function of the multivariate
symmetric stable distribution as given in Kotz, Kozubowski and
Podgorski \cite{Kotz_Kozubowski_Podgorski}, and Garoni and Frankel
\cite{Garoni_Frankel}. Instead of using \eqref{eq2} to characterize
local self-similarity, one can also use the definition of locally
asymptotically self-similar (lass) property first introduced by
Benassi, Jaffard and Roux \cite{Benassi_Jaffard_Roux} for
multifractional Brownian motion $B_{H(t)}(t)$, which is a
generalization of fractional Brownian motion with the Hurst index
replaced by the Hurst function $H(t)$, $0<H(t)<1$. It can be shown
that under some regularity conditions on $H(t)$, the multifractional
Brownian motion $B_{H(t)}(t)$  is lass. This property can be adapted
to GFGCC if we take $H(t)$ as constant with its value in (0,1).

\begin{definition}\label{def1} A stochastic process $X(t)$   is lass at a point $t_0$  with order $\kappa$ if
\begin{align}\label{eq3}\lim_{\vep\rightarrow
0^+}\left\{\frac{X(t_0+\vep
u)-X(t_0)}{\vep^{\kappa}}\right\}_{u\in\R^n} =_d T_{t_0}(u),
\end{align}and $T_{t_0}(u)$ is nontrivial. Here the convergence and equality are in the sense of finite dimensional distributions,
and $T_{t_0}(u)$ is called the tangent field of $X(t)$ at the point
$t_0$.
\end{definition} \begin{proposition}\label{p1} GFGCC is a lass random field of order $\alpha/2$; and its
tangent field is L$\acute{\text{e}}$vy fractional Brownian field of
index  $\alpha/2$.\end{proposition} \begin{proof} By using
\eqref{eq2_2}, the covariance of the increment field $\Delta_{\tau}
X_{\alpha,\beta}(t):=X_{\alpha,\beta}(t+\tau)-X_{\alpha,\beta}(t)$,
for $ \rho$, $\sigma\rightarrow 0^+$ is given by
\begin{align*}
\left\langle \Delta_\rho X_{\alpha,\beta}(t)\Delta_\sigma
X_{\alpha,\beta}(t)\right\rangle =
\beta\left(\Vert\rho\Vert^{\alpha}+\Vert\sigma\Vert^{\alpha}-\Vert\rho-\sigma\Vert^{\alpha}\right)+O\left(
\Vert \rho\Vert^{2\alpha}, \Vert \sigma\Vert^{2\alpha}\right).
\end{align*}
Let $\rho=\vep u$ and $\sigma=\vep v$, then
\begin{align*}
&\lim_{\vep\rightarrow 0^+}\left\langle \frac{\Delta_{\vep
u}X_{\alpha,\beta}(t)}{\vep^{\alpha/2}} \frac{\Delta_{\vep
v}X_{\alpha,\beta}(t)}{\vep^{\alpha/2}}\right\rangle
\\=&\lim_{\vep\rightarrow 0^+}\left\{ \beta\left(\Vert
u\Vert^{\alpha}+\Vert v\Vert^{\alpha}-\Vert
u-v\Vert^{\alpha}\right)+O(\vep^{\alpha})\right\}\\
=&2\beta\left\langle B_{\alpha/2}(u)B_{\alpha/2}(v)\right\rangle,
\end{align*}
where  $B_{\alpha/2}(u)$ is the index--$\alpha/2$
L$\acute{\text{e}}$vy fractional Brownian field with zero mean and
covariance given by
\begin{align}\label{eq4}
\left\langle
B_{\alpha/2}(u)B_{\alpha/2}(v)\right\rangle=\frac{1}{2}\left(\Vert
u\Vert^{\alpha}+\Vert v\Vert^{\alpha}-\Vert
u-v\Vert^{\alpha}\right), \hspace{1cm} u,v\in\R^n.
\end{align}
Thus up to a multiplicative constant $\sqrt{2\beta}$, the tangent
field of
 GFGCC at any point $t_0\in\R^n$   is the L$\acute{\text{e}}$vy fractional Brownian field indexed
 by $\alpha/2$.\end{proof}
The tangent field of $X_{\alpha,\beta}(t)$  at the point  $t_0$
reflects the local structure of the random field at $t_0$. In other
words, GFGCC behaves locally like a L$\acute{\text{e}}$vy fractional
Brownian field. This provides an example to the general results on
tangent fields considered by Falconer \cite{Falconer}.

The fractal dimension of the graph of a random field $X(t)$ depends
on the local property of the random field. The local irregularities
of the graph are measured by the parameter $\alpha$, which can be
regarded as the fractal index of the random field. Thus the behavior
of the covariance function at the origin to a great extent
determines the roughness of the random field. The results on the
fractal dimension of an lss field are treated by Adler \cite{Adler},
Kent and Wood \cite{Kent_Wood}, Davies and Hall \cite{Davies_Hall}.
A nice property of $\alpha/2$--lss fields is that their fractal
dimension is determined by $\alpha$.

\begin{definition} Let $X(t)$  be a stationary Gaussian field and
let $\sigma^2(\tau)=\left\langle \Delta_{\tau}X(t)^2\right\rangle$
be the variance of the increment process
$\Delta_{\tau}X(t):=X(t+\tau)-X(t)$. If there exists $\alpha\in
(0,2]$ satisfying
\begin{align}\label{eq5}
\alpha=&\sup\left\{\varsigma\,:\,
\sigma^2(\tau)=o\left(\Vert\tau\Vert^{\varsigma}\right)\;\;\text{as}\;\;\Vert\tau\Vert\rightarrow
0\right\}\\
=&\inf\left\{\varsigma\,:\,\Vert\tau\Vert^{\varsigma}=o\left(\sigma^2(\tau)\right)\;\;\text{as}\;\;\Vert\tau\Vert\rightarrow
0\right\},\nonumber
\end{align}
then $\alpha/2$ is called the fractal index of the random field
$X(t)$. Equivalently, $\alpha/2$ is  the local  H\"older index of
the random field.\end{definition} Clearly, condition \eqref{eq5} is
fulfilled by random field which satisfies \eqref{eq2} such as the
GFGCC $X_{\alpha,\beta}(t)$. Adler had shown that $\alpha/2$ is the
upper bound of the indices for which, with probability one, the
graph of $X(t)$ satisfies a global regularity of the same order.
Thus, $\alpha$ characterizes the roughness of the sample path. The
fractal dimension of GFGCC can be obtained by using the following
result for the fractal (Hausdorff) dimension of an lss field as
given in Adler \cite{Adler}, Chapter 8.

\begin{proposition}\label{p1026_4} The fractal dimension $D$ of the graph of a
locally self-similar field $X(t), t\in\R^n$, of fractal index
$\alpha/2$, over a hyperrectangle $\mathcal{C}=\prod_{i=1}^n [a_i,
b_i]$, is given by
\begin{align}\label{eq6} D=n+1-\frac{\alpha}{2}.
\end{align}\end{proposition} The estimation of $\alpha$ for lss field has been studied extensively,
see for example Wood and Chan \cite{Wood_Chan}, Istas and Lang
\cite{Istas_Lang}, Kent and Wood \cite{Kent_Wood}, Lang and Roueff
\cite{Lang_Roueff}. The parameter $\beta$ is also known as topothesy
in the studies of roughness of surfaces by Wood and Chan
\cite{Wood_Chan}, and Davies and Hall \cite{Davies_Hall}. The
topothesy of a cross section provides a measure of the roughness
which is scale--dependent in contrast to fractal dimension which is
scale--invariant.

The Gaussian random field  $X_{\alpha,\beta}(t)$ can have SRD (short
range dependence) or LRD (long range dependence), depending on the
values of the parameters $\alpha$ and $\beta$. For this purpose we
make use of the following definition which is a generalization of
the one-dimensional case considered by Flandrin et al
\cite{Flandrin_Borgnat_Amblard}, Lim and Muniandy
\cite{Lim_Muniandy}:

\begin{definition} A stationary centered Gaussian field with
covariance $C(\tau)$ is said to be a long range dependent process if
\begin{align}\label{eq7}
\int_{\R_+^n}\left|C(\tau)\right|d^n\tau=\infty.
\end{align}
Otherwise it is short range dependent. \end{definition}

\begin{proposition} \label{p2}The GFGCC $X_{\alpha,\beta}(t)$ is a long range dependent random
field if and only if $0<\alpha\beta\leq n$. \end{proposition}
\begin{proof} In order to obtain the condition for the Gaussian
random field with covariance \eqref{eq1} to be LRD, we make use of
the following integral identity given in Gradshteyn and Ryzhik
\cite{Gradshteyn_Ryzhik}, 3.251, no.11:
\begin{align*}
\int_0^{\infty}
x^{\mu-1}\left(1+x^{\rho}\right)^{-\nu}dx=\frac{1}{\rho}B\left(\frac{\mu}{\rho},
\nu-\frac{\mu}{\rho}\right),
\end{align*}
where $\rho>0$, $0<\mu<\rho v$ and
$B(x,y)=\Gamma(x)\Gamma(y)/\Gamma(x+y)$  is the beta function. Using
polar coordinates, one gets
\begin{align}\label{eq1019_1}
\int_{\R_+^n}
\left|C_{\alpha,\beta}(\tau)\right|d^n\tau=&\int_{\R_+^n}\left(1+\Vert
\tau\Vert^{\alpha}\right)^{-\beta}d^n\tau\\=&\frac{2\pi^{\frac{n}{2}}}{2^n\Gamma\left(\frac{n}{2}\right)}
\int_0^{\infty} r^{n-1}(1+r^{\alpha})^{-\beta}
dr.\nonumber\end{align} For large $r$,
$$r^{n-1}(1+r^{\alpha})^{-\beta}\sim r^{n-1-\alpha\beta}.$$
Therefore, the integral \eqref{eq1019_1} is divergent for all
$\beta>0$, $0<\alpha\beta\leq n$. For $\beta>0$ and $\alpha\beta>n$,
we have
\begin{align*}
\frac{2\pi^{\frac{n}{2}}}{2^n\Gamma\left(\frac{n}{2}\right)}
\int_0^{\infty} r^{n-1}(1+r^{\alpha})^{-\beta}
dr=&\frac{\pi^{\frac{n}{2}}}{2^{n-1}\alpha\Gamma\left(\frac{n}{2}\right)}B\left(\frac{n}{\alpha},
\beta-\frac{n}{\alpha}\right)<\infty.
\end{align*}
Therefore the condition for $X_{\alpha,\beta}(t)$  to be a Gaussian
field with LRD is $0<\alpha\beta\leq n$.\end{proof}

    The discussion above shows that it is possible to characterize the
 fractal dimension $D$ and the LRD property separately. If the covariance is re-expressed as
 $\left(1+\Vert\tau\Vert^{\alpha}\right)^{-\gamma/\alpha}$,
 which behaves like $\Vert \tau\Vert^{-\gamma}$  in the large--$\Vert\tau\Vert$ limit, then
$X_{\alpha,\gamma/\alpha}(t)$ is LRD if and only if $0<\gamma\leq
n$. Thus $\alpha$ ($0<\alpha\leq 2$) and $\gamma$ ($\gamma>0$)
respectively
  provide separate characterization of fractal dimension and LRD/SRD.
 The separate characterization of the fractal dimension (local property)
  and LRD (global property) for GFGCC appears to offer a more
   natural and flexible model than that based on a single
   parameter such as in L$\acute{\text{e}}$vy fractional Brownian
field. We note that this feature of separate characterization of
local self-similarity (hence fractal dimension) and long range
dependence is present in any stationary Gaussian field with
covariance $C(\tau)$ satisfying the asymptotic behaviors
$C(\tau)\sim A-B\Vert \tau\Vert^{\alpha}$ as $\Vert \tau \Vert
\rightarrow 0^+$, and $C(\tau) \rightarrow \Vert
\tau\Vert^{-\gamma}$ as $\Vert\tau\Vert\rightarrow \infty$, with
$\alpha\in (0,2]$, $\gamma>0$. Similarly, one can also have a
Gaussian stationary process which has separate parametrization of
fractal dimension and short range dependence \cite{a}. The ability
to have separate characterization of fractal dimension and Hurst
effect is a desirable property in the modeling of physical and
geological phenomena.

\section{  Asymptotic Properties of Spectral Density of GFGCC} \label{sec3}In this section, we consider the spectral
density of $X_{\alpha,\beta}(t)$ and its asymptotic properties.
Though the covariance of GFGCC is given by a relatively simple
expression, the analytic simplicity of the covariance function is
not inherited by the corresponding spectral density. This is similar
to the case of the stationary Gaussian field $\Xi_{\alpha,\beta}(t)$
with powered exponential covariance \eqref{eq22_1} which have simple
form, but its spectral density in general does not have closed
analytic expression. In the case of GFGCC in $\R$, a detailed study
of spectral densities (in terms of probability distributions
correspond to the characteristic functions of generalized Linnik
distributions) have been carried out by Kotz et al.
\cite{Kotz_Ostrovskii_Hayfavi} for $0<\alpha<2$, $\beta=1$, $n=1$;
by Ostrovskii \cite{Ostrovskii} for $0<\alpha<2$, $\beta=1$,
$n\in\mathbb{N}$; and by Erdogan and Ostrovskii
\cite{Erdogan_Ostrovskii} for $0<\alpha<2$, $\beta>0$, $n=1$. They
employed the contour integration representations and series
expansions of the generalized Linnik distributions. However the
techniques used in these works are less accessible to practitioners.
In this section, we derive the asymptotic properties of the spectral
densities of GFGCC for $0<\alpha\leq 2$ and $\beta>0$, which can be
regarded as an extension to the results on generalized multivariate
Linnik distributions. The techniques used in our derivations are
mathematically more tractable.

Recall that the spectral density $S(\omega)$ of a stationary field
$X(t)$ is defined as the Fourier transform of its covariance
function $C(t)=\langle X(t)X(0)\rangle$:
\begin{align*}
S(\omega)=\frac{1}{(2\pi)^{n}}\int\limits_{\R^n}e^{-i\omega.t}C(t)d^nt,
\end{align*}if the integral is convergent. If the integral does not converge, we consider $C(t)$ as
 a generalized function
and define $S(\omega)$ as the  Fourier transform of $C(t)$ in the
Schwartz space of test functions \cite{GS}. Namely, for any test
function $\psi(\omega)$ in the Schwartz class of $\R^n$, we require
\begin{align*}
\langle S(\omega), \psi(\omega)\rangle = \langle C(t),
\hat{\psi}(t)\rangle,
\end{align*}where
\begin{align*}
\hat{\psi}(t) = \frac{1}{(2\pi)^n}\int\limits_{\R^n}
e^{-i\omega.t}\psi(\omega)d^n\omega.
\end{align*}
Alternatively, the spectral density can also be defined to be the
function satisfying
\begin{align*}
C(t)=\int\limits_{\R^n} e^{i\omega.t} S(\omega)d^n\omega.
\end{align*}
For GFGCC, since \begin{align}\label{eq4_27_1} J_{\nu}(z)\sim
\sqrt{\frac{\pi}{2z}}\cos\left(z-\frac{\pi\nu}{2}-\frac{\pi}{4}\right)
\end{align}as $z\rightarrow \infty$ (\cite{AAR}, page 209), we find that when $\alpha\beta>\frac{n-1}{2}$,
it's spectral density is
\begin{align}\label{eq1031_7}
S_{\alpha,\beta}(\omega)=\frac{1}{(2\pi)^n}\int\limits_{\R^n}
\frac{e^{i\omega.t}}{(1+\Vert t\Vert^{\alpha})^{\beta}}d^nt=
\frac{\Vert\omega\Vert^{\frac{2-n}{2}}}{(2\pi)^{\frac{n}{2}}}\int_0^{\infty}
\frac{J_{\frac{n-2}{2}}(\Vert\omega\Vert
t)}{(1+t^{\alpha})^{\beta}}t^{\frac{n}{2}}dt,
\end{align}where $J_{\nu}(z)$ is the Bessel function.    When $\alpha=2$, using
no.4 of 6.565 in  \cite{Gradshteyn_Ryzhik}, we have the explicit
formula
\begin{align}\label{eq1031_4}
S_{2,\beta}(\omega)=\frac{\Vert
\omega\Vert^{\beta-\frac{n}{2}}}{2^{\frac{n}{2}+\beta-1}\pi^{\frac{n}{2}}\Gamma(\beta)}K_{\frac{n}{2}-\beta}
(\Vert\omega\Vert),
\end{align}if $\beta>(n-1)/4$. Here $K_{\nu}(z)$ is the modified Bessel function. The
formula no.7 of 6.576 in  \cite{Gradshteyn_Ryzhik} shows that if
$\beta\in (0,n)$, then
\begin{align*} &\int\limits_{\R^n}e^{i\omega.t}\left(\frac{\Vert
\omega\Vert^{\beta-\frac{n}{2}}}{2^{\frac{n}{2}+\beta-1}\pi^{\frac{n}{2}}\Gamma(\beta)}K_{\frac{n}{2}-\beta}
(\Vert\omega\Vert)\right)d^n\omega\\=&(2\pi)^{\frac{n}{2}}\int_0^{\infty}
\frac{J_{\frac{n-2}{2}}(\omega \Vert t\Vert)}{(\omega \Vert
t\Vert)^{\frac{n-2}{2}}}\left(\frac{
\omega^{\beta-\frac{n}{2}}}{2^{\frac{n}{2}+\beta-1}\pi^{\frac{n}{2}}\Gamma(\beta)}K_{\frac{n}{2}-\beta}
( \omega )\right)\omega^{n-1}d\omega=\left(1+\Vert
t\Vert^{2}\right)^{-\beta}.
\end{align*}Therefore \eqref{eq1031_4} is still  the spectral density  when $\beta\in(0,(n-1)/4]$. For
general $\alpha<2$, no  explicit formula such as \eqref{eq1031_4}
can be found for $S_{\alpha,\beta}(\omega)$. When $n=1$, the formula
\eqref{eq1031_7} gives the spectral density of $X_{\alpha,\beta}(t)$
for all values of $\alpha\in (0,2]$ and $\beta>0$. For $n\geq 2$, we
would also like to find a formula for the spectral density that is
valid for all $\alpha\in (0,2)$ and $\beta>0$. For this purpose, it
would be beneficial to investigate the case $n=1$ first. When $n=1$,
we can rewrite \eqref{eq1031_7} as
\begin{align*}
S_{\alpha,\beta}(\omega)=\frac{1}{\pi}\text{Re}\,\int_0^{\infty}\frac{e^{i|\omega|
t}}{(1+t^{\alpha})^{\beta}}dt.
\end{align*}Let $$f(\zeta)=\frac{e^{i|\omega|\zeta}}{(1+\zeta^{\alpha})^{\beta}}, \hspace{1cm}
-\pi<\text{arg}\, \zeta\leq \pi,$$ and consider the region
$\mathfrak{D}_r$ in the complex plane defined by
$$\mathfrak{D}_r=\Bigl\{ z\in \C\,:\, |z|\leq r, \,\text{Re}\,z>0,
\,\text{Im}\,z>0\Bigr\}.$$When $\alpha\in(0,2)$, the function $f$ is
an analytic function on the domain $\mathfrak{D}_r$. Therefore, by
Cauchy integral formula,
\begin{align}\label{eq1031_6}
\oint\limits_{\pa\mathfrak{D}_r}f(\zeta)d\zeta=0.
\end{align}Notice that the boundary of $\mathfrak{D}_r$,
$\pa\mathfrak{D}_r$, consists of three components: the line segment
$l_{r,1}$ along the real axis from $0$ to $r$, the arc $C_r$ of the
circle $|z|=r$ from $r$ to $ir$, and the line segment $l_{r,2}$
along the imaginary axis from $ir$ to $0$. On the arc $C_r$, if
$r>1$, then
\begin{align*}
|f(\zeta)|\leq \frac{e^{-|\omega| \text{Im}\,
\zeta}}{(r^{\alpha}-1)^{\beta}}.
\end{align*}Therefore
\begin{align*}
\lim_{r\rightarrow\infty}\int_{C_r}f(\zeta)d\zeta=0,
\end{align*}and \eqref{eq1031_6} implies that
\begin{align*}
\lim_{r\rightarrow \infty}\int_{l_{r,1}} f(\zeta)d\zeta
=-\lim_{r\rightarrow \infty}\int_{l_{r,2}}f(\zeta) d\zeta.
\end{align*}This gives us
\begin{align}\label{eq1101_2}
S_{\alpha,\beta}(\omega) =&
\frac{1}{\pi}\text{Re}\,\int_0^{\infty}\frac{e^{i|\omega|t}dt}{(1+t^{\alpha})^{\beta}}=
-\frac{1}{\pi}\text{Im}\,\int_0^{\infty}
\frac{e^{-|\omega|u}}{\left(1+e^{\frac{i\pi\alpha}{2}}u^{\alpha}\right)^{\beta}}du.
\end{align}
 For $n\geq 2$, we can derive a formula similar to \eqref{eq1101_2}. Recall that
the Hankel's function of the first kind $H_{\nu}^{(1)}(z)$ is
defined as $$H_{\nu}^{(1)}(z)=J_{\nu}(z)+iN_{\nu}(z),$$ where
$N_{\nu}(z)$ is the modified Bessel function of the second kind or
called the Neumann function. Using Hankel's function, we can rewrite
\eqref{eq1031_7} as
\begin{align*}
S_{\alpha,\beta}(\omega)=
\frac{\Vert\omega\Vert^{\frac{2-n}{2}}}{(2\pi)^{\frac{n}{2}}}\text{Re}\,\int_0^{\infty}
\frac{H_{\frac{n-2}{2}}^{(1)}(\Vert\omega\Vert
t)}{(1+t^{\alpha})^{\beta}}t^{\frac{n}{2}}dt.
\end{align*}For $z\rightarrow \infty$, we have
(\cite{Gradshteyn_Ryzhik}, no.3 of 8.451)
\begin{align*}
H_{\nu}^{(1)}(z)\sim \sqrt{\frac{2}{\pi
z}}\exp\left\{i\left(z-\frac{\pi\nu}{2}-\frac{\pi}{4}\right)\right\}.
\end{align*}Therefore, we can show as in the $n=1$ case that
\begin{align*}
S_{\alpha,\beta}(\omega)=-\frac{\Vert\omega\Vert^{\frac{2-n}{2}}}{(2\pi)^{\frac{n}{2}}}\text{Im}\,\int_0^{\infty}
\frac{H_{\frac{n-2}{2}}^{(1)}(i\Vert\omega\Vert
u)}{\left(1+e^{\frac{i\pi\alpha}{2}}u^{\alpha}\right)^{\beta}}(iu)^{\frac{n}{2}}du.
\end{align*}Using the formula (\cite{Gradshteyn_Ryzhik},  no.1 of 8.407)
\begin{align*}
K_{\nu}(z)=\frac{i\pi}{2}e^{\frac{i\nu\pi}{2}}H_{\nu}^{(1)}(iz),
\end{align*} we have
finally
\begin{align}\label{eq1031_8}
S_{\alpha,\beta}(\omega)=-\frac{\Vert\omega\Vert^{\frac{2-n}{2}}}{2^{\frac{n-2}{2}}\pi^{\frac{n+2}{2}}}\text{Im}\,\int_0^{\infty}
\frac{K_{\frac{n-2}{2}}(\Vert\omega\Vert
u)}{\left(1+e^{\frac{i\pi\alpha}{2}}u^{\alpha}\right)^{\beta}}u^{\frac{n}{2}}du.
\end{align}This formula agrees with the formula for multivariate Linnik
distribution proved in \cite{Ostrovskii} for $\alpha\in (0,2),
\beta=1$ and $n\in\mathbb{N}$. Notice that the right hand side of
\eqref{eq1031_8} is well-defined for all $\alpha,\beta>0$. Using the
formula (no.2 of 6.521 in \cite{Gradshteyn_Ryzhik}),
\begin{align*}
\int_0^{\infty}
xK_{\nu}(ax)J_{\nu}(bx)dx=\frac{b^{\nu}}{a^{\nu}(a^2+b^2)},
\hspace{1cm}\nu>-1,
\end{align*}we have
\begin{align*}
&\int\limits_{\R^n} e^{i\omega.t} \left\{-\frac{ \Vert\omega\Vert
^{\frac{2-n}{2}}}{2^{\frac{n-2}{2}}\pi^{\frac{n+2}{2}}}\text{Im}\,\int_0^{\infty}
\frac{K_{\frac{n-2}{2}}( \Vert\omega\Vert
u)}{\left(1+e^{\frac{i\pi\alpha}{2}}u^{\alpha}\right)^{\beta}}u^{\frac{n}{2}}du\right\}d^n\omega\\=&(2\pi)^{\frac{n}{2}}\Vert
t\Vert^{\frac{2-n}{2}}\int_0^{\infty} J_{\frac{n-2}{2}}(\omega \Vert
t\Vert)\omega^{\frac{n}{2}}\left\{-\frac{ \omega
^{\frac{2-n}{2}}}{2^{\frac{n-2}{2}}\pi^{\frac{n+2}{2}}}\text{Im}\,\int_0^{\infty}
\frac{K_{\frac{n-2}{2}}( \omega
u)}{\left(1+e^{\frac{i\pi\alpha}{2}}u^{\alpha}\right)^{\beta}}u^{\frac{n}{2}}du\right\}d\omega\\
=&-\frac{2\Vert
t\Vert^{\frac{2-n}{2}}}{\pi}\text{Im}\,\int_0^{\infty}\frac{u^{\frac{n}{2}}}
{\left(1+e^{\frac{i\pi\alpha}{2}}u^{\alpha}\right)^{\beta}}\int_0^{\infty}\omega
K_{\frac{n-2}{2}}(\omega u)J_{\frac{n-2}{2}}(\omega \Vert
t\Vert)d\omega du\\
=&-\frac{2}{\pi}\text{Im}\,\int_0^{\infty}\frac{u}
{\left(1+e^{\frac{i\pi\alpha}{2}}u^{\alpha}\right)^{\beta}(u^2+\Vert
t\Vert^2)}du\end{align*}
\begin{align*}
=&-\frac{1}{\pi i}\int_{-\infty}^{\infty}\frac{u}
{\left(1+e^{\frac{i\pi\alpha}{2}}u^{\alpha}\right)^{\beta}(u^2+\Vert
t\Vert^2)}du.\end{align*} When $\alpha\in (0,2)$, residue calculus
implies that this last integral is equal to \begin{align*}&2
\text{Res}_{u=-i\Vert t\Vert}\frac{u}{(u-i\Vert
t\Vert)\left(1+e^{\frac{i\pi\alpha}{2}}u^{\alpha}\right)^{\beta}}=\frac{1}{(1+\Vert
t\Vert^{\alpha})^{\beta}}.
\end{align*}This shows that \eqref{eq1031_8} is indeed the spectral
density of GFGCC for all $\alpha\in (0,2)$ and $\beta>0$. We would
also like to remark that although the formula \eqref{eq1031_8} is
derived under the assumption $n\geq 2$, but since
$$K_{-1/2}(z)=\sqrt{\frac{\pi}{2z}}e^{-z},$$ therefore when $n=1$, the
formula \eqref{eq1031_8} reduces to the formula \eqref{eq1101_2}. We
summarize the result as follows.

\begin{proposition}If $\alpha\in (0,2)$ and $\beta>0$, the spectral density of the GFGCC
 $X_{\alpha,\beta}(t)$ is given by
\begin{align}\label{eq1031_8_2}
S_{\alpha,\beta}(\omega)=
-\frac{\Vert\omega\Vert^{\frac{2-n}{2}}}{2^{\frac{n-2}{2}}\pi^{\frac{n+2}{2}}}\text{Im}\,\int_0^{\infty}
\frac{K_{\frac{n-2}{2}}(\Vert\omega\Vert
u)}{\left(1+e^{\frac{i\pi\alpha}{2}}u^{\alpha}\right)^{\beta}}u^{\frac{n}{2}}du.
\end{align}If $\alpha=2$ and $\beta>0$, the spectral density of the
GFGCC $X_{2,\beta}(t)$ is given by
\begin{align}\label{eq1031_4_2}
S_{2,\beta}(\omega)=\frac{\Vert
\omega\Vert^{\beta-\frac{n}{2}}}{2^{\frac{n}{2}+\beta-1}\pi^{\frac{n}{2}}\Gamma(\beta)}K_{\frac{n}{2}-\beta}
(\Vert\omega\Vert).
\end{align}

\end{proposition}

To find the high frequency  behavior of the spectral density, we
first consider the case where $\alpha=2$. Using the fact that
(\cite{Gradshteyn_Ryzhik}, no.6 of 8.451) as $z\rightarrow \infty$,
\begin{align*}
K_{\nu}(z)\sim
\sqrt{\frac{\pi}{2z}}e^{-z}\sum_{j=0}^{\infty}\frac{1}{(2z)^j}\frac{\Gamma\left(\nu+j+\frac{1}{2}\right)
}{j!\Gamma\left(\nu-j+\frac{1}{2}\right)}.
\end{align*}This implies that as $\Vert\omega\Vert\rightarrow
\infty$,
\begin{align}\label{eq1112_11}
S_{2,\beta}(\omega)\sim
\frac{\Vert\omega\Vert^{\beta-\frac{n+1}{2}}}{2^{\frac{n-1}{2}+\beta}\pi^{\frac{n-1}{2}}\Gamma(\beta)}e^{-\Vert
\omega\Vert}\sum_{j=0}^{\infty}\frac{1}{(2\Vert\omega\Vert)^j}\frac{\Gamma\left(j-\beta+\frac{n+1}{2}\right)
}{j!\Gamma\left(\frac{n+1}{2}-j-\beta\right)},
\end{align}with leading term
\begin{align*}
S_{2,\beta}(\omega)\sim
\frac{\Vert\omega\Vert^{\beta-\frac{n+1}{2}}}{2^{\frac{n-1}{2}+\beta}\pi^{\frac{n-1}{2}}\Gamma(\beta)}e^{-\Vert
\omega\Vert}.
\end{align*}

For general $\alpha\in (0,2)$, to find the high frequency behavior
of $S_{\alpha,\beta}(\omega)$, we make use of eq.
\eqref{eq1031_8_2}. Making a change of variable and using
\begin{align}\label{eq1112_1}
\frac{1}{\left(1+e^{\frac{i\pi\alpha}{2}}\frac{u^{\alpha}}{\Vert\omega\Vert^{\alpha}}\right)^{\beta}}=
\sum_{j=0}^{m}\frac{(-1)^j}{j!}\frac{\Gamma(\beta+j)}{\Gamma(\beta)}
e^{\frac{i\pi\alpha j}{2}}\frac{u^{\alpha
j}}{\Vert\omega\Vert^{\alpha j}}+O(\Vert
\omega\Vert^{-\alpha(m+1)}),\end{align} as
$\Vert\omega\Vert\rightarrow \infty$, we find that
\begin{align*}
S_{\alpha,\beta}(\omega)=&
-\frac{\Vert\omega\Vert^{-n}}{2^{\frac{n-2}{2}}\pi^{\frac{n+2}{2}}}
\text{Im}\,\int_0^{\infty}
 K_{\frac{n-2}{2}}(
u)\frac{u^{\frac{n}{2}}}{\left(1+e^{\frac{i\pi\alpha}{2}}\frac{u^{\alpha}}{\Vert\omega\Vert^{\alpha}}\right)^{\beta}}du
\\=&-\frac{\Vert\omega\Vert^{-n}}{2^{\frac{n-2}{2}}\pi^{\frac{n+2}{2}}}
\frac{1}{\Gamma(\beta)}\text{Im}\,\int_0^{\infty}
 K_{\frac{n-2}{2}}(
u) \sum_{j=0}^{m} \frac{(-1)^j}{j!}\Gamma(\beta+j) e^{\frac{i\pi
\alpha j}{2}}\frac{u^{\alpha j}}{\Vert \omega\Vert^{\alpha
j}}u^{\frac{n}{2}}
du\\
&+O\left(\Vert\omega\Vert^{-\alpha(m+1)-n}\right)\hspace{5cm}\text{as}\;\;\Vert\omega\Vert\rightarrow
\infty.
\end{align*}Using the formula (\cite{Gradshteyn_Ryzhik}, no.16 of 6.561)
\begin{align*}
\int_0^{\infty}
x^{\mu}K_{\nu}(x)dx=2^{\mu-1}\Gamma\left(\frac{1+\mu+\nu}{2}\right)\Gamma\left(\frac{1+\mu-\nu}{2}\right),
\hspace{1cm}\text{Re}\, (\mu+1-|\nu|)>0,
\end{align*}and the definition of asymptotic expansion (\cite{AAR}, page 611), we
 conclude that as $\Vert\omega\Vert \rightarrow \infty$,
$S_{\alpha,\beta}(\omega)$ behaves asymptotically as
\begin{align}\label{eq1112_10}
\frac{1}{ \pi^{\frac{n+2}{2}}}
\frac{1}{\Gamma(\beta)}\sum_{j=1}^{\infty}\frac{(-1)^{j-1}2^{\alpha
j }}{j!}\Gamma(\beta+j)\Gamma\left(\frac{\alpha
j+n}{2}\right)\Gamma\left(\frac{\alpha
j+2}{2}\right)\sin\frac{\pi\alpha j}{2}\Vert\omega\Vert^{-\alpha
j-n}.
\end{align}
Notice that there is a drastic change of  high frequency limit of
$S_{\alpha,\beta}(\omega)$ when $\alpha<2$ and $\alpha=2$. In fact,
naively  putting $\alpha=2$ in \eqref{eq1112_10}   give identically
zero terms. This is a hint that as $\Vert\omega\Vert\rightarrow
\infty$, $S_{2,\beta}(\omega)$ does not have  polynomial decay,
instead it decays exponentially as is verified by \eqref{eq1112_11}.
We summarize the results as follows:
\begin{proposition}\label{p1113_1}
If $\alpha\in (0,2)$ and $\beta>0$, the high frequency limit of the
spectral density $S_{\alpha,\beta}(\omega)$ is given by the
following asymptotic series
\begin{align}\label{eq1112_2}
S_{\alpha,\beta}(\omega)\sim\frac{1}{ \pi^{\frac{n+2}{2}}}
\frac{1}{\Gamma(\beta)}\sum_{j=1}^{\infty}\frac{(-1)^{j-1}2^{\alpha
j }}{j!}\Gamma(\beta+j)\Gamma\left(\frac{\alpha
j+n}{2}\right)\Gamma\left(\frac{\alpha
j+2}{2}\right)\sin\frac{\pi\alpha j}{2}\Vert\omega\Vert^{-\alpha
j-n}.
\end{align}If $\alpha=2$ and $\beta>0$, the high frequency limit of the
spectral density $S_{2,\beta}(\omega)$ is given by the following
asymptotic series
\begin{align*}S_{2,\beta}(\omega)\sim
\frac{\Vert\omega\Vert^{\beta-\frac{n+1}{2}}}{2^{\frac{n-1}{2}+\beta}\pi^{\frac{n-1}{2}}\Gamma(\beta)}e^{-\Vert
\omega\Vert}\sum_{j=0}^{\infty}\frac{1}{(2\Vert\omega\Vert)^j}\frac{\Gamma\left(j-\beta+\frac{n+1}{2}\right)
}{j!\Gamma\left(\frac{n+1}{2}-j-\beta\right)}.\end{align*}
\end{proposition}When $\alpha\in (0,2),\beta=1, n\in\mathbb{N}$ and
$\alpha\in (0,2), \beta>0, n=1$, \eqref{eq1112_2} agrees with the
results given in   \cite{Ostrovskii} and  \cite{Erdogan_Ostrovskii}
respectively. In particular, we observe that when $\alpha\in (0,2)$,
the high frequency behavior of the spectral  density of GFGCC is
\begin{align}\label{eq1112_3}S_{\alpha,\beta}(\omega)\sim
\frac{2^{\alpha}\beta}{\pi^{\frac{n+2}{2}}}\Gamma\left(
\frac{\alpha+n}{2}\right)\Gamma\left(\frac{\alpha+2}{2}\right)\sin\frac{\pi\alpha}{2}\Vert
\omega\Vert^{-\alpha-n}\rightarrow 0^+,
\hspace{1cm}\Vert\omega\Vert\rightarrow \infty,\end{align}which is
independent of $\beta$. Kent and Wood \cite{Kent_Wood} have shown
that if a random field has spectral density satisfying
\eqref{eq1112_3}, then its covariance satisfies \eqref{eq2} with
locally self--similar property. However, the converse is not true.

In this connection we remark that for the Gaussian stationary field
$\Xi_{\alpha,\beta}(t)$ with powered exponential covariance which is
lss with
\begin{align}\label{eq22}
\left\langle
\Xi_{\alpha,\beta}(t+\tau)\Xi_{\alpha,\beta}(t)\right\rangle
=e^{-\beta\Vert\tau\Vert^{\alpha}}=1-\beta\Vert\tau\Vert^{\alpha}\left[1+O(\Vert\tau\Vert^{\alpha})\right],
\hspace{1cm}\Vert\tau\Vert\rightarrow 0^+,
\end{align}its small $\Vert\tau\Vert$ behavior has a similar form
as that of GFGCC \eqref{eq2_2}. Thus it is not surprising that this
two random fields have the same tail behavior for their spectral
densities at high frequencies as given by \eqref{eq1112_3}. The
detailed calculation carried out by Garoni and Frankel
\cite{Garoni_Frankel} for the probability distribution of the
multivariate L$\acute{\text{e}}$vy stable distribution with
characteristic function given by \eqref{eq22} confirms this. We also
note that \eqref{eq1112_2} can be used to verify that the tangent
field at any point $t_0$ has spectral density which varies as $\Vert
\omega\Vert^{-\alpha-n}$ for $\Vert \omega\Vert\rightarrow \infty$.
If we let $\alpha=2H$, then the tangent field is just the
L$\acute{\text{e}}$vy fractional Brownian field in $\R^n$. Such a
relationship can be viewed as a consequence of the
Tauberian--Abelian theorem (see e.g. \cite{Korevaar}).

 For the low frequency  behavior of the spectral density $S_{\alpha,\beta}(\omega)$, we
 first consider the case $\alpha=2$.  Using 8.485, 8.445, 8.446 of
\cite{Gradshteyn_Ryzhik}, we find that if $\nu\notin\Z$,
\begin{align}\label{eq1031_2}
K_{\nu}(z)=K_{-\nu}(z)=\frac{\pi}{2\sin (\pi
\nu)}\left\{\sum_{j=0}^{\infty}
\frac{(z/2)^{2j-\nu}}{j!\Gamma(j+1-\nu)}-\sum_{j=0}^{\infty}\frac{(z/2)^{2j+\nu}}{j!
\Gamma(j+1+\nu)}\right\};
\end{align}whereas when $\nu=\pm m$, where $m$ is a nonnegative
integer,
\begin{align}\label{eq1031_3}
K_{\nu}(z)=&\frac{1}{2}\sum_{j=0}^{m-1}\frac{(-1)^j
(m-j-1)!}{j!}\left(\frac{z}{2}\right)^{2j-m}
\\&+(-1)^{m+1}\sum_{j=0}^{\infty}\frac{(z/2)^{m+2j}}
{j!(m+j)!}\left\{\ln\frac{z}{2}-\frac{1}{2}\psi(j+1)-\frac{1}{2}\psi(j+1+m)\right\}.\nonumber
\end{align}Here $\psi(z)=\Gamma'(z)/\Gamma(z)$ is the logarithm
derivative of the Gamma function. Therefore from \eqref{eq1031_4_2},
 we find that:

\vspace{0.2cm}\noindent $\bullet$\;\; if $\beta>n/2$, then as
$\Vert\omega\Vert\rightarrow 0^+$,
\begin{align}\label{eq1112_7}
S_{2,\beta}(\omega)\sim
\frac{\Gamma\left(\beta-\frac{n}{2}\right)}{2^n\pi^{\frac{n}{2}}\Gamma(\beta)}
.
\end{align}

\vspace{0.2cm}\noindent $\bullet$\;\; if $\beta=n/2$, then as
$\Vert\omega\Vert\rightarrow 0^+$,
\begin{align}\label{eq1112_8}
S_{2,\beta}(\omega)\sim &
\frac{1}{2^{n-1}\pi^{\frac{n}{2}}\Gamma\left(
\frac{n}{2}\right)}\left\{- \ln \Vert\omega\Vert +\ln
2+\psi(1)\right\}
\\
=&\frac{1}{2^{n-1}\pi^{\frac{n}{2}}\Gamma\left(
\frac{n}{2}\right)}\left\{-\ln \Vert \omega\Vert +\ln
2-\gamma\right\}
\nonumber,
\end{align}where $\gamma$ is the Euler constant.

\vspace{0.2cm}\noindent $\bullet$\;\; if $\beta<n/2$, then as
$\Vert\omega\Vert\rightarrow 0^+$,
\begin{align}\label{eq1112_9}
S_{2,\beta}(\omega)\sim\frac{\Gamma\left(\frac{n}{2}-\beta\right)}{2^{2\beta}\pi^{\frac{n}{2}}
\Gamma(\beta)}\Vert\omega\Vert^{2\beta-n}.
\end{align}

\vspace{0.5cm}In fact, by considering the cases
$\beta-\frac{n}{2}\in \Z$ and $\beta-\frac{n}{2}\neq \Z$ separately
and substituting the series \eqref{eq1031_2} and \eqref{eq1031_3}
into \eqref{eq1031_4_2}, we can express the spectral density
$S_{2,\beta}(\omega)$ in terms of convergent power series in
$\Vert\omega\Vert$.

\begin{figure}\centering \epsfxsize=.8\linewidth  \epsffile{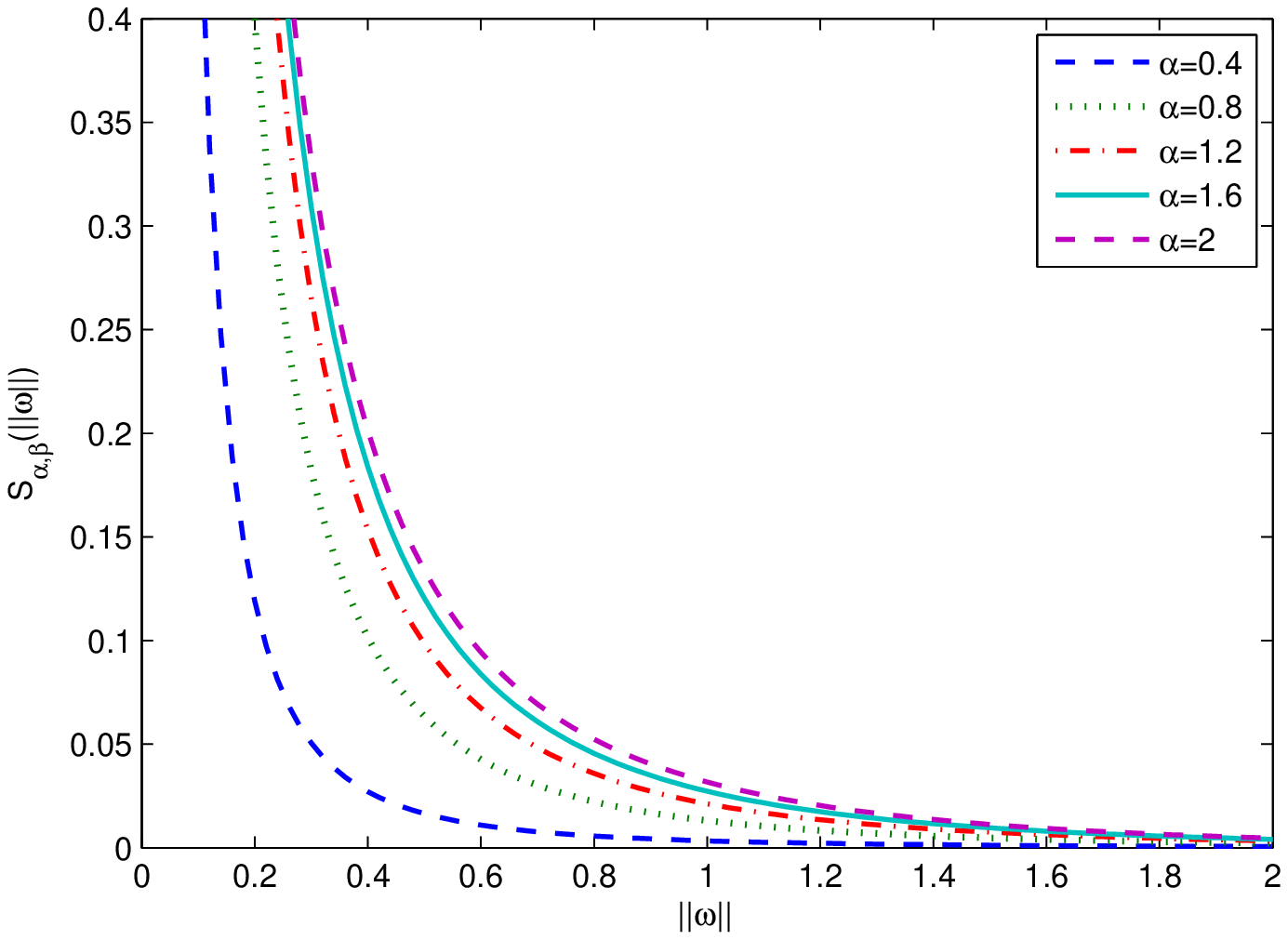} \centering
\epsfxsize=.8\linewidth \epsffile{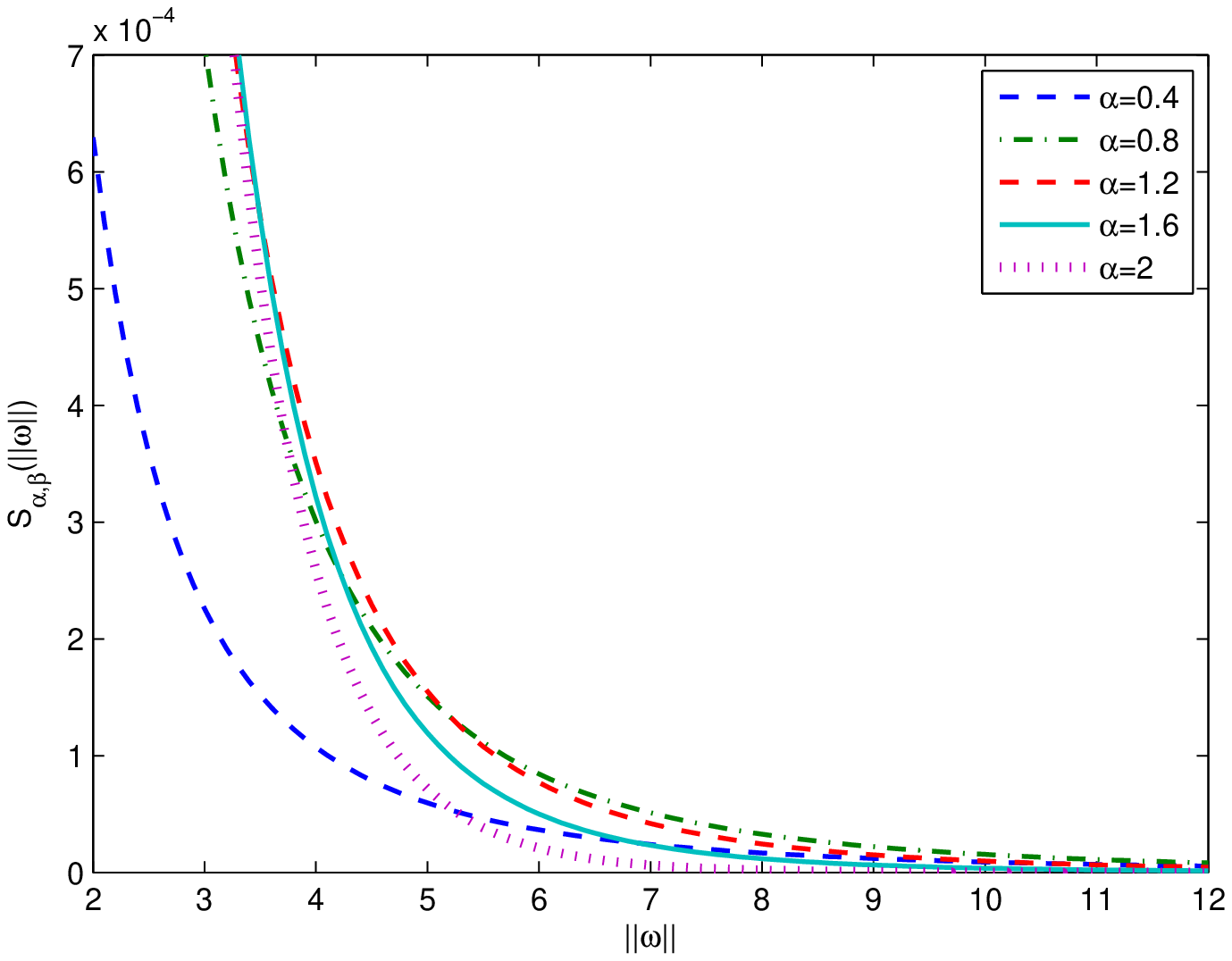}\caption{The spectral
density $S_{\alpha,\beta}(\Vert\omega\Vert)$ as a function of
$\Vert\omega\Vert$ when $n=3$ and $\alpha\beta=1.5$.}\end{figure}

\begin{figure}\centering \epsfxsize=.8\linewidth  \epsffile{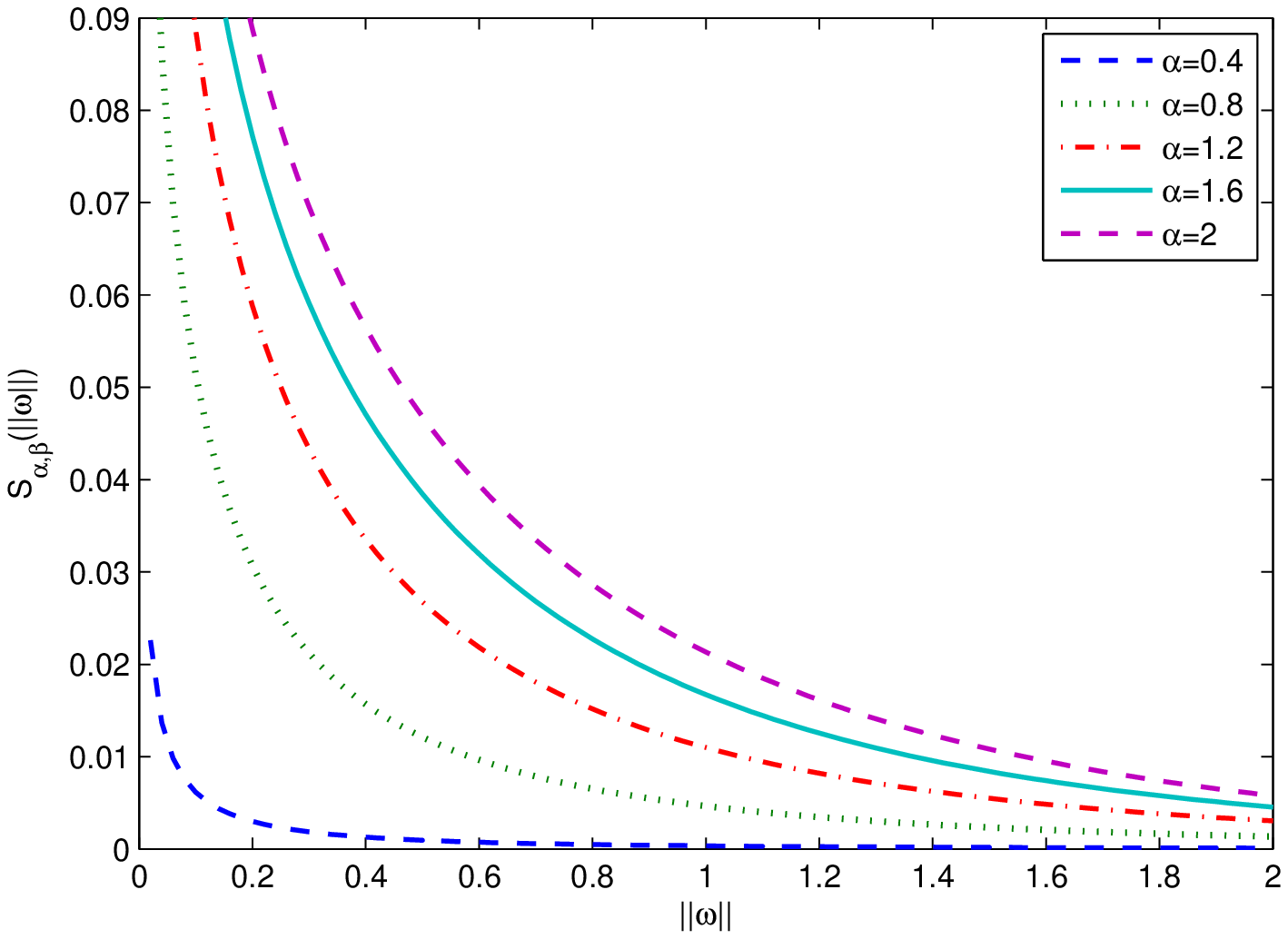} \centering
\epsfxsize=.8\linewidth \epsffile{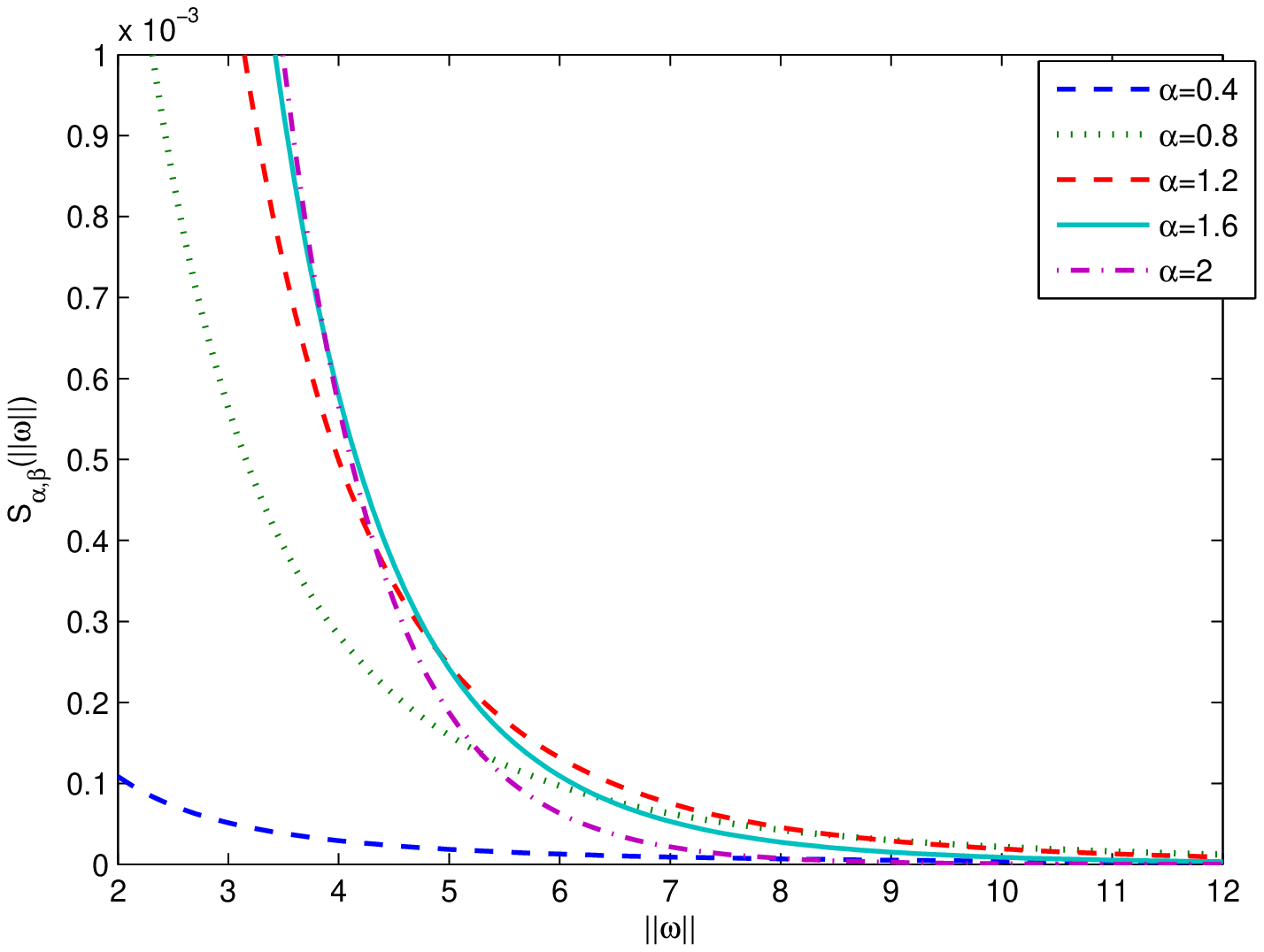}\caption{The spectral
density $S_{\alpha,\beta}(\Vert\omega\Vert)$ as a function of
$\Vert\omega\Vert$ when $n=3$ and $\alpha\beta=3$. }\end{figure}

\begin{figure}\centering \epsfxsize=.8\linewidth  \epsffile{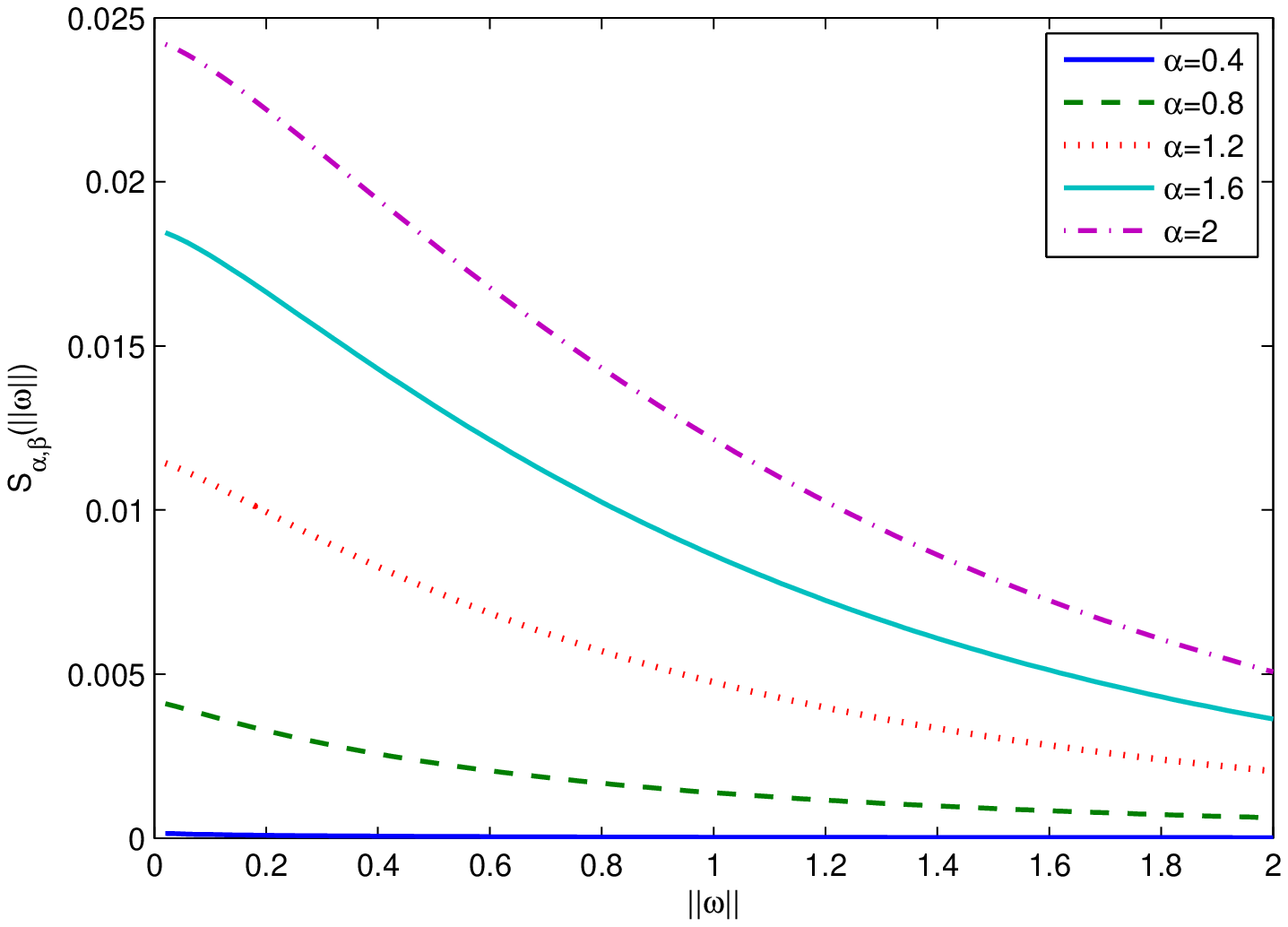} \centering
\epsfxsize=.8\linewidth \epsffile{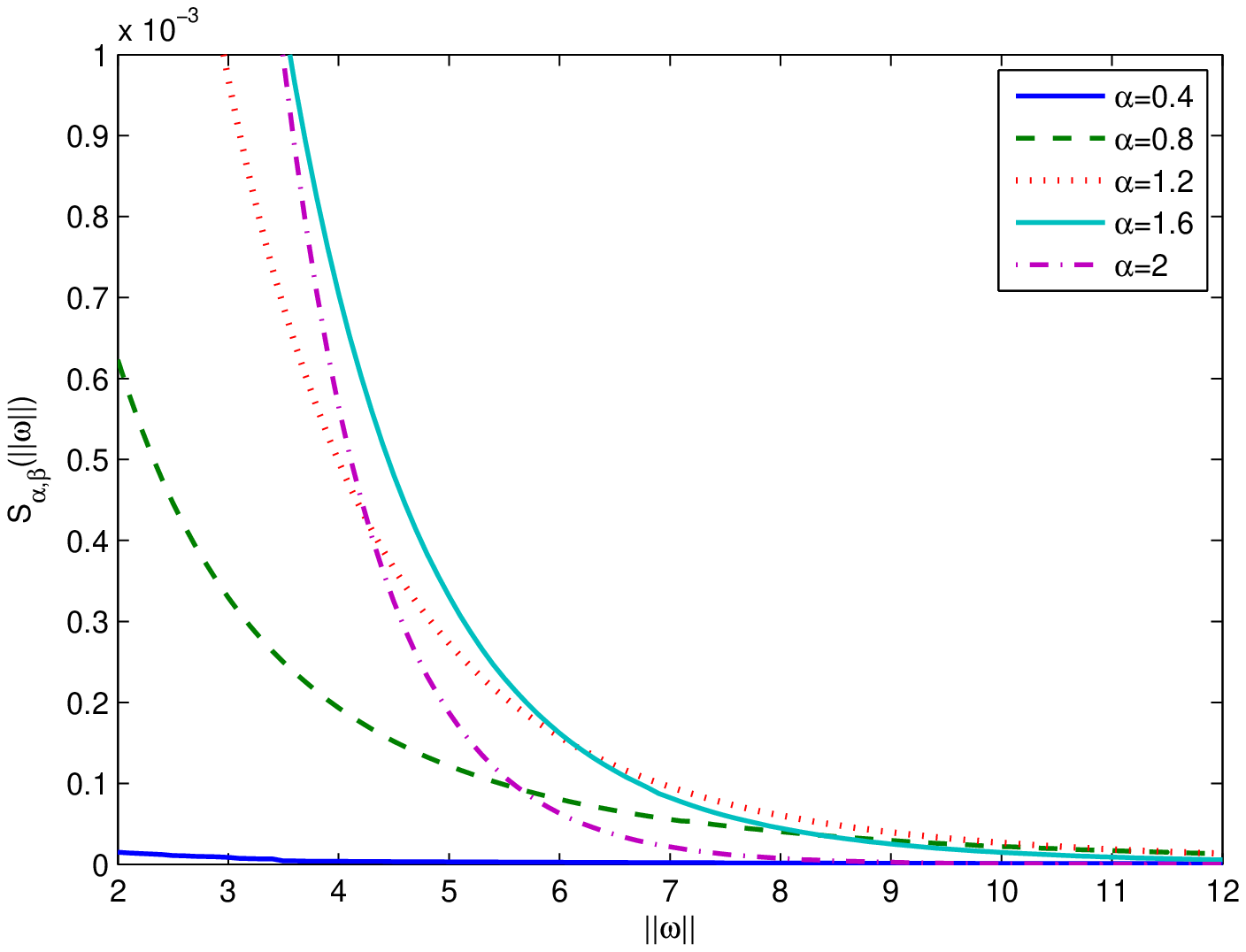}\caption{The spectral
density $S_{\alpha,\beta}(\Vert\omega\Vert)$ as a function of
$\Vert\omega\Vert$ when $n=3$ and $\alpha\beta=4.5$. $\alpha=0.4,
0.8,1.2, 1.6, 2$ for $S_1, S_2, S_3, S_4, S_5$
respectively.}\end{figure}

For general $\alpha\in (0,2)$,  the low frequency behavior of
$S_{\alpha,\beta}(\omega)$ depends on the arithmetic nature of
$\alpha$  and $\beta$. The method we are going to employ does not
allow the derivation of the whole asymptotic series as obtained by
Kotz et al \cite{Kotz_Ostrovskii_Hayfavi}, Erdogan and Ostrovskii
\cite{Ostrovskii, Erdogan_Ostrovskii}. We will only derive the
leading behavior of the spectral density $S_{\alpha,\beta}(\omega)$,
which only depends  on the algebraic conditions $\alpha\beta>n$,
$\alpha\beta=n$ or $\alpha\beta<n$.  These conditions are less
stringent than the arithmetic conditions considered in
\cite{Kotz_Ostrovskii_Hayfavi, Ostrovskii, Erdogan_Ostrovskii}.
However,  the simpler method employed here  provides the necessary
$S_{\alpha\beta}(\omega)$, $\Vert\omega\Vert\rightarrow 0^+$
asymptotic behaviors which are sufficient for most practical
purposes.

 When $\alpha\beta>n$, since (\cite{Gradshteyn_Ryzhik},
Eq. 8.402)
\begin{align}\label{eq1102_2}J_{\nu}(z)=\frac{z^{\nu}}{2^{\nu}\Gamma(\nu+1)}+O\left(z^{\nu+2}\right)
\hspace{1cm}\text{as}\; z\rightarrow 0,\end{align}we find from
\eqref{eq1031_7} that as $\Vert \omega\Vert\rightarrow 0^+$,
\begin{align}\label{eq1112_4}
S_{\alpha,\beta}(\omega)\sim\frac{1}{2^{n-1}\pi^{\frac{n}{2}}\Gamma\left(\frac{n}{2}\right)}\int_0^{\infty}
\frac{t^{n-1}}{(1+t^{\alpha})^{\beta}}dt=\frac{1}{2^{n-1}\pi^{\frac{n}{2}}\Gamma\left(\frac{n}{2}\right)}\frac
{\Gamma\left(\frac{n}{\alpha}\right)\Gamma\left(\beta-\frac{n}{\alpha}\right)}{\alpha\Gamma(\beta)},
\end{align}(\cite{Gradshteyn_Ryzhik}, no.11 of 3.251). When
$\alpha\beta=n$, we make a change of variable on \eqref{eq1031_7} to
get
\begin{align*}
S_{\alpha,\beta}(\omega)=\frac{1}{(2\pi)^{\frac{n}{2}}}\int_0^{\infty}
J_{\frac{n-2}{2}}(t)\frac{t^{\frac{n}{2}}}{(\Vert
\omega\Vert^{\alpha}+t^{\alpha})^{\beta}}dt.
\end{align*}In view of the leading behavior of the Bessel function
$J_{\nu}(z)$ as $z\rightarrow 0$ \eqref{eq1102_2}, we write
$S_{\alpha,\beta}(\omega)$ as the sum of two terms
$S_{\alpha,\beta}^1(\omega)$ and $S_{\alpha,\beta}^2(\omega)$ where
\begin{align*}
S_{\alpha,\beta}^1(\omega)=&\frac{1}{(2\pi)^{\frac{n}{2}}}\int_0^1\left(J_{\frac{n-2}{2}}(t)-
\frac{t^{\frac{n-2}{2}}}{2^{\frac{n-2}{2}}\Gamma\left(\frac{n}{2}\right)}\right)\frac{t^{\frac{n}{2}}}{
(\Vert\omega\Vert^{\alpha}+t^{\alpha})^{\beta}}dt\\
&+\frac{1}{(2\pi)^{\frac{n}{2}}}\int_1^{\infty}
J_{\frac{n-2}{2}}(t)\frac{t^{\frac{n}{2}}}{(\Vert
\omega\Vert^{\alpha}+t^{\alpha})^{\beta}}dt,
\end{align*}and
\begin{align*}
S_{\alpha,\beta}^2(\omega)=\frac{1}{2^{n-1}\pi^{\frac{n}{2}}\Gamma\left(\frac{n}{2}\right)}
\int_0^1\frac{t^{n-1}}{(\Vert\omega\Vert^{\alpha}+t^{\alpha})^{\beta}}dt.
\end{align*}As $\Vert \omega\Vert\rightarrow 0^+$,
\eqref{eq4_27_1} and \eqref{eq1102_2} show that
$S_{\alpha,\beta}^1(\omega)$ has a finite limit. By Lebesgue's
dominated convergence theorem, the limit is given by
$S_{\alpha,\beta}^1(0)$. Namely
\begin{align*}
S_{\alpha,\beta}^1(\omega)\xrightarrow {\Vert\omega\Vert\rightarrow
0^+}&S_{\alpha,\beta}^1(0)\\=&\frac{1}{(2\pi)^{\frac{n}{2}}}\int_0^1\left(J_{\frac{n-2}{2}}(t)-
\frac{t^{\frac{n-2}{2}}}{2^{\frac{n-2}{2}}\Gamma\left(\frac{n}{2}\right)}\right)t^{-\frac{n}{2}}dt
+\frac{1}{(2\pi)^{\frac{n}{2}}}\int_1^{\infty} J_{\frac{n-2}{2}}(t)
t^{-\frac{n}{2}}dt.
\end{align*}This expression can be evaluated using regularization method.
More precisely, using Lebesgue's dominated convergence theorem
again, we find that
\begin{align*}
S_{\alpha,\beta}^1(0)=&\lim_{\vep\rightarrow
0^+}\left\{\frac{1}{(2\pi)^{\frac{n}{2}}}\int_0^1\left(J_{\frac{n-2}{2}}(t)-
\frac{t^{\frac{n-2}{2}}}{2^{\frac{n-2}{2}}\Gamma\left(\frac{n}{2}\right)}\right)t^{-\frac{n}{2}+\vep}dt
+\frac{1}{(2\pi)^{\frac{n}{2}}}\int_1^{\infty} J_{\frac{n-2}{2}}(t)
t^{-\frac{n}{2}+\vep}dt\right\}
\\=&
\lim_{\vep\rightarrow
0^+}\left\{\frac{1}{(2\pi)^{\frac{n}{2}}}\int_0^{\infty}
J_{\frac{n-2}{2}}(t)t^{-\frac{n}{2}+\vep}dt
-\frac{1}{2^{n-1}\pi^{\frac{n}{2}}\Gamma\left(\frac{n}{2}\right)}\int_0^1t^{-1+\vep}dt\right\}.
\end{align*}The formula no.14 of 6.561 in  \cite{Gradshteyn_Ryzhik} then gives
\begin{align*}
S_{\alpha,\beta}^1(0)=&\lim_{\vep\rightarrow
0^+}\left\{\frac{2^{\vep}\Gamma\left(\frac{\vep}{2}\right)}{2^n\pi^{\frac{n}{2}}\Gamma\left(\frac{n-\vep}{2}\right)
}-\frac{1}{2^{n-1}\pi^{\frac{n}{2}}\Gamma\left(\frac{n}{2}\right)}\frac{1}{\vep}\right\}\\
=&\lim_{\vep\rightarrow
0^+}\frac{1}{2^{n-1}\pi^{\frac{n}{2}}\Gamma\left(\frac{n}{2}\right)}\frac{1}{\vep}\left\{(1+\vep\ln
2)\left(1+\frac{\vep}{2}\psi(1)\right)\left(1+\frac{\vep}{2}\psi\left(\frac{n}{2}\right)\right)-1\right\}\\
=&\frac{1}{2^{n}\pi^{\frac{n}{2}}\Gamma\left(\frac{n}{2}\right)}\left\{2\ln
2+\psi(1)+\psi\left(\frac{n}{2}\right) \right\}.
\end{align*}For the term $S_{\alpha,\beta}^2(\omega)$, we make a change of variable $u=t^{\alpha}$ or
equivalently $t=u^{\frac{\beta}{n}}$, to get
\begin{align*}
S_{\alpha,\beta}^2(\omega)=&\frac{\beta}{2^{n-1}\pi^{\frac{n}{2}}n\Gamma\left(\frac{n}{2}\right)}
\int_0^1\frac{u^{\beta-1}du}{(\Vert\omega\Vert^{\alpha}+u)^{\beta}}.\end{align*}We
split  $S_{\alpha,\beta}^2(\omega)$ again into a sum of two terms
$S_{\alpha,\beta}^3(\omega)$ and $S_{\alpha,\beta}^4(\omega)$,
where\begin{align*}S_{\alpha,\beta}^3(\omega)=&\frac{\beta}{2^{n-1}\pi^{\frac{n}{2}}n\Gamma\left(\frac{n}{2}\right)}
\int_0^1\frac{1}{(\Vert\omega\Vert^{\alpha}+u)}du\\=&
\frac{\beta}{2^{n-1}\pi^{\frac{n}{2}}n\Gamma\left(\frac{n}{2}\right)}\ln\frac{1+\Vert
\omega|\Vert^{\alpha}}{\Vert\omega\Vert^{\alpha}}\sim
\frac{1}{2^{n-1}\pi^{\frac{n}{2}}\Gamma\left(\frac{n}{2}\right)}\ln
\frac{1}{\Vert\omega\Vert}+O(\Vert\omega\Vert^{\alpha}),
\end{align*}and\begin{align*}S_{\alpha,\beta}^4(\omega)=
&\frac{\beta}{2^{n-1}\pi^{\frac{n}{2}}n\Gamma\left(\frac{n}{2}\right)}\int_0^1\frac{u^{\beta-1}-
(\Vert\omega\Vert^{\alpha}+u)^{\beta-1}}{(\Vert\omega\Vert^{\alpha}+u)^{\beta}}du\\
=&\frac{\beta}{2^{n-1}\pi^{\frac{n}{2}}n\Gamma\left(\frac{n}{2}\right)}
\int_0^{\frac{1}{\Vert\omega\Vert^{\alpha}}}\frac{u^{\beta-1}-(1+u)^{\beta-1}}{(1+u)^{\beta}}du.
\end{align*}When $\Vert\omega\Vert\rightarrow 0^+$,
\begin{align*}
S_{\alpha,\beta}^4(\omega)\sim
 \frac{\beta}{2^{n-1}\pi^{\frac{n}{2}}n\Gamma\left(\frac{n}{2}\right)}\int_0^{\infty}
 \frac{u^{\beta-1}-(1+u)^{\beta-1}}{(1+u)^{\beta}}du.
\end{align*}The integral is a convergent integral with value given
by (\cite{Gradshteyn_Ryzhik}, 3.219 page 316)
\begin{align*}
\int_0^{\infty}\frac{u^{\beta-1}-(1+u)^{\beta-1}}{(1+u)^{\beta}}du
=-\psi(\beta)-\gamma.
\end{align*} Putting everything together, we find  that
\begin{align*}
S_{\alpha,\beta}^2(\omega)\sim
\frac{1}{2^{n-1}\pi^{\frac{n}{2}}\Gamma\left(\frac{n}{2}\right)}\left\{
\ln\frac{1}{\Vert\omega\Vert}-\frac{\beta}{n}\left(\psi(\beta)+\gamma\right)\right\}\hspace{1cm}
\text{as}\;\;\Vert\omega\Vert\rightarrow 0^+.
\end{align*}Therefore, as $\Vert\omega\Vert \rightarrow 0^+$,
\begin{align}\label{eq1112_5}
S_{\alpha,\beta}(\omega)\sim\frac{1}{2^{n-1}\pi^{\frac{n}{2}}\Gamma\left(\frac{n}{2}\right)}\left\{
\ln\frac{1}{\Vert\omega\Vert}-\frac{\beta}{n}\left(\psi(\beta)+\gamma\right)+\ln
2-\frac{1}{2}\gamma+\frac{1}{2}\psi\left(\frac{n}{2}\right)\right\}.
\end{align}
When $\alpha\beta<n$, care has to be taken since \eqref{eq1031_7} is
defined only for $\alpha\beta>\frac{n-1}{2}$. Making a change of
variable, one finds that
\begin{align*}
S_{\alpha,\beta}(\omega)=
\frac{\Vert\omega\Vert^{\alpha\beta-n}}{(2\pi)^{\frac{n}{2}}}\int_0^{\infty}\frac{J_{\frac{n-2}{2}}(u)
u^{\frac{n}{2}}}{(\Vert\omega\Vert^{\alpha}+u^{\alpha})^{\beta}}du.
\end{align*}It is easy to
 verify that for any $\alpha>0$ and $\beta>0$,
 when $\Vert\omega\Vert\rightarrow 0^+$,
\begin{align*}
S_{\alpha,\beta}(\omega)\sim
\frac{\Vert\omega\Vert^{\alpha\beta-n}}{(2\pi)^{\frac{n}{2}}}\lim_{\vep\rightarrow
0^+}\lim_{\Vert\omega\Vert\rightarrow
0}\int_0^{\infty}\frac{J_{\frac{n-2}{2}}(u)e^{-\vep
u}u^{\frac{n}{2}}}{(\Vert\omega\Vert^{\alpha}+u^{\alpha})^{\beta}}du.
\end{align*}Now using no.1 of 6.621 in \cite{Gradshteyn_Ryzhik}, we have
\begin{align*}
\lim_{\Vert\omega\Vert\rightarrow
0}\int_0^{\infty}\frac{J_{\frac{n-2}{2}}(u)e^{-\vep
u}u^{\frac{n}{2}}}{(\Vert\omega\Vert^{\alpha}+u^{\alpha})^{\beta}}du=&\int_0^{\infty}
J_{\frac{n-2}{2}}(u)e^{-\vep u}u^{\frac{n}{2}-\alpha\beta}\\
=&\frac{1}{2^{\frac{n-2}{2}}}\frac{\Gamma(n-\alpha\beta)}{\sqrt{(\vep^2+1)^{n-\alpha\beta}}\Gamma\left(
\frac{n}{2}\right)}\,_2F_1\left(\frac{n-\alpha\beta}{2},
\frac{\alpha\beta-1}{2}; \frac{n}{2}; \frac{1}{1+\vep^2}\right),
\end{align*}where $_2F_1(a,b;c;z)$ is the hypergeometric function
\begin{align*}
_2F_1(a,b;c;z)=\sum_{j=0}^{\infty}\frac{(a)_j(b)_j}{(c)_j}\frac{z^j}{j!},\hspace{1cm}
(x)_j:=x(x+1)\ldots(x+j-1)=\frac{\Gamma(x+j)}{\Gamma(x)}.
\end{align*}Using the formula no.2 of 9.131 in \cite{Gradshteyn_Ryzhik}, we find
that
\begin{align*}
&\,_2F_1\left(\frac{n-\alpha\beta}{2}, \frac{\alpha\beta-1}{2};
\frac{n}{2};
\frac{1}{1+\vep^2}\right)=\frac{\Gamma\left(\frac{n}{2}\right)\Gamma\left(\frac{1}{2}\right)}
{\Gamma\left(\frac{\alpha\beta}{2}\right)\Gamma\left(\frac{n+1-\alpha\beta}{2}\right)}\,
_2F_1\left(\frac{n-\alpha\beta}{2}, \frac{\alpha\beta-1}{2};
\frac{1}{2};\frac{\vep^2}{1+\vep^2}\right)\\
&\hspace{3cm}+\left(\frac{\vep^2}{1+\vep^2}\right)^{\frac{1}{2}}\frac{\Gamma\left(\frac{n}{2}\right)\Gamma
\left(-\frac{1}{2}\right)}{\Gamma\left(\frac{n-\alpha\beta}{2}\right)
\Gamma\left(\frac{\alpha\beta-1}{2}\right)}\,_2F_1\left(\frac{\alpha\beta}{2},\frac{n-\alpha\beta+1}{2};
\frac{3}{2};\frac{\vep^2}{1+\vep^2}\right).
\end{align*}Since
\begin{align*}
\lim_{z\rightarrow 0}\,_2F_1(a,b;c;z)=1,
\end{align*}therefore,
\begin{align*}
\lim_{\vep\rightarrow 0^+}\lim_{\Vert\omega\Vert\rightarrow
0}\int_0^{\infty}\frac{J_{\frac{n-2}{2}}(u)e^{-\vep
u}u^{\frac{n}{2}}}{(\Vert\omega\Vert^{\alpha}+u^{\alpha})^{\beta}}du=&\frac{\sqrt{\pi}}
{2^{\frac{n-2}{2}}}\frac{\Gamma(n-\alpha\beta)}{\Gamma\left(\frac{\alpha\beta}{2}\right)\Gamma\left(
\frac{n+1-\alpha\beta}{2}\right)}
=2^{\frac{n}{2}-\alpha\beta}\frac{\Gamma\left(\frac{n-\alpha\beta}{2}\right)}
{\Gamma\left(\frac{\alpha\beta}{2}\right)},
\end{align*}where we have used the formula $\Gamma(2z)=2^{2z-1}\pi^{-\frac{1}{2}}\Gamma(z)
\Gamma(z+(1/2))$. This gives for $0<\alpha\beta<n$,
\begin{align}\label{eq1112_6}
S_{\alpha,\beta}(\omega)\sim\frac{\Vert\omega\Vert^{\alpha\beta-n}}{2^{\alpha\beta}\pi^{\frac{n}{2}}}
\frac{\Gamma\left(\frac{n-\alpha\beta}{2}\right)}
{\Gamma\left(\frac{\alpha\beta}{2}\right)}\hspace{1cm}\text{as}\;\Vert\omega\Vert\rightarrow
0^+.
\end{align}It is easy to verify that by putting $\alpha=2$ in \eqref{eq1112_4},
\eqref{eq1112_5} and \eqref{eq1112_6}, we get back \eqref{eq1112_7},
\eqref{eq1112_8} and \eqref{eq1112_9} respectively. The low
frequency behaviors of $S_{\alpha,\beta}(\omega)$ are summarized
below.
\begin{proposition}\label{p1113_2}For all $\alpha\in (0,2]$ and $\beta>0$, the
low frequency limit of the spectral density
$S_{\alpha,\beta}(\omega)$  is given by
\begin{align*}
S_{\alpha,\beta}(\omega)\sim&\frac{\Vert\omega\Vert^{\alpha\beta-n}}{2^{\alpha\beta}\pi^{\frac{n}{2}}}
\frac{\Gamma\left(\frac{n-\alpha\beta}{2}\right)}
{\Gamma\left(\frac{\alpha\beta}{2}\right)},\hspace{2cm}\text{if}\;\;\alpha\beta<n;\\
S_{\alpha,\beta}(\omega)\sim&\frac{1}{2^{n-1}\pi^{\frac{n}{2}}\Gamma\left(\frac{n}{2}\right)}\left\{
\ln\frac{1}{\Vert\omega\Vert}-\frac{\beta}{n}\left(\psi(\beta)+\gamma\right)+\ln
2-\frac{1}{2}\gamma+\frac{1}{2}\psi\left(\frac{n}{2}\right)\right\},\hspace{0.2cm}\text{if}\;\;\alpha\beta=n;\\
S_{\alpha,\beta}(\omega)\sim&\frac{1}{2^{n-1}\pi^{\frac{n}{2}}\Gamma\left(\frac{n}{2}\right)}\frac
{\Gamma\left(\frac{n}{\alpha}\right)\Gamma\left(\beta-\frac{n}{\alpha}\right)}{\alpha\Gamma(\beta)}
,\hspace{1cm}\text{if}\;\;\alpha\beta>n.
\end{align*}

\end{proposition}
We see that if $0<\alpha\beta\leq n$, $S_{\alpha,\beta}(\omega)$ is
divergent at the origin and if $0<\alpha\beta<n$,
$S_{\alpha,\beta}(\omega)\sim \Vert \omega\Vert^{\alpha\beta-n}$ as
$\Vert\omega\Vert\rightarrow 0^+$. Note that the condition
$0<\alpha\beta\leq n$ agrees with the LRD condition. In fact, it is
a basic fact that for a stationary field with positive covariance,
it is LRD if and only if its spectral density diverges at the
origin. We would also like to point out that the low frequency limit
of the spectral density of the Gaussian field with powered exponent
covariance $\Xi_{\alpha,\beta}(t)$ \eqref{eq22_1} is finite as shown
by Garoni and Frankel \cite{Garoni_Frankel}, in agreement with the
fact that $\Xi_{\alpha,\beta}(t)$ is SRD.

We  remark that by considering
$$S_{\alpha,\beta}(\omega)-\text{"leading order term as}\;
\Vert\omega\Vert\rightarrow 0^+\text{"},$$ we can find the next
order term in the asymptotic expansion of $S_{\alpha,\beta}(\omega)$
when  $\Vert\omega\Vert\rightarrow 0^+$ using the same methods we
employed above. The results  depend on more complicated conditions
on $\alpha$ and $\beta$.

We also briefly remark that the asymptotic behavior of the spectral
density at low frequency is connected to the large time behavior of
the covariance function. For the covariance function
$C_{\alpha,\beta}(\tau)$ which satisfies \begin{align}\label{eq39}
C_{\alpha,\beta}(\tau)\sim L(\tau)\Vert\tau\Vert^{-\alpha\beta},
\hspace{1cm}\Vert\tau\Vert\rightarrow \infty,
\end{align}where $L(\tau)$ is a slowly
varying function for large $\Vert\tau\Vert$, i.e.
$L(c\tau)/L(\tau)\rightarrow 1$ as $\Vert\tau\Vert\rightarrow
\infty$ for all positive constant $c$,
Hardy--Littlewood--Karamata--Tauberian theorem (see e.g. \cite{Leo1,
Leo2}) implies that the spectral density $S_{\alpha,\beta}(\omega)$
has the following asymptotic behavior
\begin{align*}
S_{\alpha,\beta}(\omega)\sim c_{n,\alpha\beta}
\Vert\omega\Vert^{\alpha\beta-n}L\left(\frac{1}{\omega}\right),\hspace{1cm}\Vert\omega\Vert\rightarrow
0^+, \hspace{1cm}\text{if}\;\; \alpha\beta<n,
\end{align*}where
\begin{align*} c_{n,\alpha\beta}=\frac{1}{2^{\alpha\beta}\pi^{\frac{n}{2}}}
\frac{\Gamma\left(\frac{n-\alpha\beta}{2}\right)}
{\Gamma\left(\frac{\alpha\beta}{2}\right)}.
\end{align*}

 \section{Gaussian Sheet with Generalized Cauchy Covariance}
We can also introduce another type of $n$-dimensional process
$X^{\#}_{\alpha,\beta}(t)$, called Gaussian sheet with generalized
Cauchy covariance (GSGCC).

\begin{definition}The GSGCC is a centered Gaussian random field
$X^{\#}_{\alpha,\beta}(t)$ indexed by two multidimensional
parameters $\alpha=(\alpha_1,\ldots,\alpha_n)$ and
$\beta=(\beta_1,\ldots,\beta_n)$, with $\alpha_i\in (0,2]$,
$\beta_i>0$ for all $i=1,\ldots, n$, and with covariance   given by
\begin{align}\label{eq41}
C^{\#}_{\alpha,\beta}(\tau)=&\left\langle
X^{\#}_{\alpha,\beta}(t+\tau)X^{\#}_{\alpha,\beta}(t)\right\rangle \\
=&\prod_{i=1}^n C_{\alpha_i,\beta_i}(\tau_i)=\prod_{i=1}^n
\left(1+|\tau_i|^{\alpha_i}\right)^{-\beta_i},\nonumber
\end{align}where $C_{\alpha_i,\beta_i}(\tau_i)$  is the covariance of one-dimensional GFGCC  indexed
by $\alpha_i$ and $\beta_i$.\end{definition} Heuristically, one can
regard $X^{\#}_{\alpha,\beta}(t)$ as the product of $n$ independent
one-dimensional GFGCC processes $X_{\alpha_i,\beta_i}(t_i)$.
However, since the product of independent normal random variables is
not a normal random variable, we shall not write
$X^{\#}_{\alpha,\beta}(t)=\prod_{i=1}^n X_{\alpha_i,\beta_i}(t_i)$.
Nevertheless, it is true that for any integer $m\geq 1$,
$$E\left(\left[X^{\#}_{\alpha,\beta}(t)\right]^m\right)=\prod_{i=1}^n
E\left(\left[X_{\alpha_i,\beta_i}(t_i)\right]^m\right).$$

The GSGCC   $X^{\#}_{\alpha,\beta}(t)$ is an anisotropic Gaussian
random field and it
  does not satisfy the definition of local
self--similarity given in Definition \ref{def2}. However, we can
show that it is lass and has tangent field according to Definition
\ref{def1}. Assume that
$\alpha_1=\ldots=\alpha_{m_{\alpha}}<\alpha_{m_{\alpha}+1}\leq\ldots\leq
\alpha_n$. Namely, for some $m_{\alpha}\in\{1,\ldots, n\}$,
$\alpha_1=\alpha_2=\ldots=\alpha_{m_{\alpha}}=\min\alpha$, and for
all $i\geq m_{\alpha}+1$, $\alpha_i\gneq \min\alpha$. Here
$\min\alpha=\min\{\alpha_1,\ldots, \alpha_n\}.$

\begin{proposition}\label{p1026_2}GSGCC is lass with tangent field $T_{\alpha,\beta}(t)$, which is
a stationary centered Gaussian field with covariance
\begin{align*}
\left\langle T_{\alpha,\beta}(u)T_{\alpha,\beta}(v)\right\rangle
=\sum_{i=1}^{m_{\alpha}}
\beta_i\left(|u_i|^{\min\alpha}+|v_i|^{\min\alpha}-|u_i-v_i|^{\min
\alpha}\right).
\end{align*}\end{proposition}
\begin{proof}From definition, it is easy to see that\begin{align}\label{eq42}
C^{\#}_{\alpha,\beta}(\tau)=\prod_{i=1}^n
\left\{1-\beta_i|\tau_i|^{\alpha_i}\left[1+O|\tau_i|^{\alpha_i}\right]\right\}
\hspace{0.5cm}\text{as}\hspace{0.5cm}|\tau_i|\rightarrow 0^+,
\;i=1,\ldots,n.
 \end{align} Therefore,
\begin{align}\label{eq1022_1}
C^{\#}_{\alpha,\beta}(\tau)=1-\sum_{i=1}^{m_{\alpha}}\beta_i|\tau_i|^{\min\alpha}+O\left(\sum_{i=m_{\alpha}+1}^n
|\tau_i|^{\alpha_i}+\sum_{i=1}^{m_{\alpha}}|\tau_i|^{2\min\alpha}\right).
\end{align}Consequently,
\begin{align*}
&\lim_{\vep\rightarrow 0^+}\left\langle \frac{\Delta_{\vep
u}X^{\#}_{\alpha,\beta}(t)}{\vep^{\min\alpha/2}}\frac{\Delta_{\vep
v}X^{\#}_{\alpha,\beta}(t)}{\vep^{\min\alpha/2}}\right\rangle\\
=&\lim_{\vep\rightarrow
0^+}\frac{1}{\vep^{\min\alpha}}\Bigl\{C^{\#}_{\alpha,\beta}(\vep
(u-v))-C^{\#}_{\alpha,\beta}(\vep
u)-C^{\#}_{\alpha,\beta}(\vep v)+1\Bigr\}\\
=&\sum_{i=1}^{m_{\alpha}}
\beta_i\left(|u_i|^{\min\alpha}+|v_i|^{\min\alpha}-|u_i-v_i|^{\min
\alpha}\right).
\end{align*}

\end{proof}From this proposition we see that in general, the tangent field
$T_{\alpha,\beta}(u)$  of the GSGCC $X^{\#}_{\alpha,\beta}(t)$
defined by Definition \ref{def1} are uncorrelated in some of the
directions of $\R^n$. It doesn't fully capture the locally
self--similar property of $X^{\#}_{\alpha,\beta}(t)$. There is
another measure of local self--similarity, called local multiple
self-similarity that better describe this kind of process. Recall
that \cite{Genton_Perrin_Taqqu} a stochastic process $X(t)$ is
called multi-self--similar (mss) of index $H\in \R_+^n$ if and only
if for any $c\in \R_+^n$,
$$X\left(c_1t_1,\ldots, c_n t_n\right)=_d c_1^{H_1}\ldots
c_{n}^{H_n} X(t_1,\ldots, t_n).$$ It is obvious that a $H$-mss field
is a $\left[\sum_{i=1}^n H_i\right]$-ss field, but in general a
self--similar field is not multi--self--similar. The notion of lss
can be generalized accordingly.
\begin{definition}
A centered stationary Gaussian field  is $n$-ple locally
self--similar of order $\alpha/2$ if its covariance function
$C(\tau)$ satisfies for $\Vert \tau\Vert\rightarrow 0^+$,
\begin{align*}
C(\tau)=\prod_{i=1}^n
\left(A_i-B_i|\tau_i|^{\alpha_i}\left[1+O\left(|\tau_i|^{\delta_i}\right)\right]\right),
\end{align*}for some $A_i, B_i,\delta_i>0$, $1\leq i\leq n$.
\end{definition}
 In view of \eqref{eq42}, it is easy to see that $X^{\#}_{\alpha,\beta}(t)$ is
 $n$-ple locally self--similar with $A_i=1$, $B_i=\beta_i$ and
 $\delta_i=\alpha_i$.
We also note that the local multi-self--similarity is equivalent to
\begin{align}\label{eq43}
X^{\#}_{\alpha,\beta}(t+c\tau_ie_i)-X^{\#}_{\alpha,\beta}(t)=_d
c^{\alpha_i}\left[X^{\#}_{\alpha,\beta}(t+\tau_ie_i)-X^{\#}_{\alpha,\beta}(t)\right]\;\;\text{as}\;\;
\tau_i\rightarrow 0^+,
\end{align} for any $c\in \R^+$. Here $e_i$ is the unit vector in the $t_i$ direction.
This can also be rephrased as $n$-ple lass. \begin{proposition}
$X^{\#}_{\alpha,\beta}(t)$ is $n$-ple lass. More precisely, for
every $i=1,\ldots, n$,
\begin{align}\label{eq1019_3}
\lim_{\vep\rightarrow 0^+}\left\{\frac{X^{\#}_{\alpha,\beta}(t+\vep
u_ie_i)-X^{\#}_{\alpha,\beta}(t)}{\vep^{\alpha_i/2}_i}\right\}_{u_i\in\R}=\sqrt{2\beta}B_{\alpha_i/2}(u_i),
\end{align}where $B_{\alpha/2}(u)$ is the one dimensional
fractional Brownian motion of index $\alpha/2$.
\end{proposition}The proof is the same as Proposition \ref{p1}.

One should observe that the $n$-ple lass here is different from the
lass defined in Definition \ref{def1}. We take the limit process in
each of the $t_i$--direction separately and the limit process  can
be regarded as the partial derivative process of
$X^{\#}_{\alpha,\beta}(t)$ in the $t_i$-direction. A disadvantage of
this definition is that the limit process of an $n$-dimensional
field is a one-dimensional process.
 Therefore, we will
define another limit process that will better reflect local
multi-self--similarity, as were considered in \cite{Herbin} and
\cite{Ayache_Leger}. Given a stochastic process $X(t), t\in \R^n$,
we define its total increment by $u\in\R^n$ at the point $t\in\R^n$
as
\begin{align*}
\square_{u} X(t)=\sum_{\delta\in \{0,1\}^n} (-1)^{n-\sum_{i=1}^n
\delta_i}X\left(t+\sum_{i=1}^n \delta_i u_i e_i\right).
\end{align*}When $n=1$, this is the same as the increment, i.e.
$\square_u X(t)=\Delta_u X(t)$. When $n=2$, we have
\begin{align*}
\square_{(u_1, u_2)}X(t_1, t_2)=X(t_1+u_1, t_2+u_2)-X(t_1+u_1,
t_2)-X(t_1, t_2+u_2)+X(t_1, t_2),
\end{align*}which is also known as the rectangular increment. For general $n$, given $t, u\in \R^n$, the set of points
$$\left\{t+\sum_{i=1}^n \delta_i u_i
e_i\right\}_{\delta\in\{0,1\}^n}$$ are the  vertices of the
hyperrectangle with $t$ and $t+u$ as one of the main diagonals. By
giving the vertex $t+u$ a weight $+1$, the other vertices of the
hyperrectangle can be given a weight $+1$ and $-1$ alternatingly so
that adjacent vertices has different weights. $\square_u X(t)$ is
then the weighted sum of the field $X(t)$ at the vertices of the
hyperrectangle.

Recall that the fractional Brownian sheet $B_{\alpha/2}^{\#}(t),
t\in \R^n$ of index $\alpha/2\in (0,1)^n$ is a centered Gaussian
process with covariance
\begin{align*}
\left\langle B_{\alpha/2}^{\#}(t)B_{\alpha/2}^{\#}(s)\right\rangle
=\frac{1}{2^n}\prod_{i=1}^n
\left(|t_i|^{\alpha_i}+|s_i|^{\alpha_i}-|t_i-s_i|^{\alpha_i}\right).
\end{align*}It is well known that $B_{\alpha/2}^{\#}(t)$ is a
self--similar field of order $\sum_{i=1}^n \alpha_i/2$, and
multi--self--similar of order $\alpha/2$. However, unlike the
L$\acute{\text{e}}$vy fractional Brownian field \eqref{eq4}, the
fractional Brownian sheet is not a process with stationary
increment. Nevertheless, it is a process with stationary total
increment, i.e.,
\begin{align*}
\left\{ \square_u B_{\alpha/2}^{\#}(t),\; u\in\R^n\right\}=_d\left\{
\square_u B_{\alpha/2}^{\#}(s),\; u\in\R^n\right\}, \hspace{1cm}
\forall\; t,s\in\R^n.
\end{align*}

Now returning to the local asymptotic multi--self--similarity of
GSGCC, we can show the following:
\begin{proposition}\label{p1026_5}\begin{align*}
\lim_{\vep \rightarrow 0^+}\left\langle \frac{\square_{\vep.u}
X^{\#}_{\alpha,\beta}(t)}{\prod_{i=1}^n \vep_i^{\alpha_i/2}
}\right\rangle_{u\in\R^n}=_d \left[\prod_{i=1}^n
\sqrt{2\beta_i}\right] B_{\alpha/2}^{\#}(u).
\end{align*}Here $\vep.u=\sum_{i=1}^n \vep_i u_i e_i$.
\end{proposition}\begin{proof}
If we heuristically write $X^{\#}_{\alpha,\beta}(t)$ as
$\prod_{i=1}^nX_{\alpha_i,\beta_i}(t_i)$, then heuristically the
total increment $\square_u X^{\#}_{\alpha,\beta}(t)$ can be written
as the product of increments $\prod_{i=1}^n
\Delta_{u_i}X_{\alpha_i,\beta_i}(t_i)$, and the result follows from
Proposition \ref{p1}.

For a more rigorous proof, notice that
\begin{align*}
\left\langle \square_{\vep.u} X^{\#}_{\alpha,\beta}(t)
\square_{\vep.v} X^{\#}_{\alpha,\beta}(t)\right\rangle
=&\sum_{\delta\in \{0,1\}^n}\sum_{\eta\in\{0,1\}^n}(-1)^{
\sum_{i=1}^n \delta_i+\sum_{i=1}^n \eta_i}
C^{\#}_{\alpha,\beta}\left(\sum_{i=1}^n\vep_i\left[\delta_i
u_i-\eta_i v_i\right]e_i\right)\\
=&\sum_{\delta\in \{0,1\}^n}\sum_{\eta\in\{0,1\}^n}(-1)^{
\sum_{i=1}^n \delta_i+\sum_{i=1}^n \eta_i} \prod_{i=1}^n
C_{\alpha_i,\beta_i}\left(\vep_i \left[\delta_i u_i-\eta_i
v_i\right]
\right)\\
=&\prod_{i=1}^n
\Bigl\{C_{\alpha_i,\beta_i}(0)-C_{\alpha_i,\beta_i}(\vep_i
u_i)-C_{\alpha_i,\beta_i}(\vep_i
v_i)+C_{\alpha_i,\beta_i}(\vep_i[u_i-v_i])\Bigr\}\\
=&\prod_{i=1}^n \left\langle \Delta_{\vep_i
u_i}X_{\alpha_i,\beta_i}(t_i)\Delta_{\vep_i
v_i}X_{\alpha_i,\beta_i}(t_i)\right\rangle.
\end{align*}The result follows.
\end{proof}From this proposition, we see that the total increment
process can capture the locally multi-self--similar property of the
GSGCC better. It also shows that locally, GSGCC behaves similarly as
the fractional Brownian sheet. One would  tend to use Proposition
\ref{p1026_2} and Proposition \ref{p1026_4} to conclude that the
Hausdorff dimension of the graph of  GSGCC over a hyperrectangle is
$$ d_{n,\alpha}:=n+1-\frac{1}{2}\min\{\alpha_1,\ldots,
\alpha_n\}.$$ However, Proposition \ref{p1026_4} cannot be applied
here since Proposition \ref{p1026_2} does not imply that the fractal
index of GSGCC is $\min\alpha$. For the fractional Brownian sheet
$B_{\alpha/2}^{\#}$, Kamont \cite{Kamont} showed that the Hausdorff
dimension of its graph   is bounded above by $d_{n,\alpha}$. In
\cite{Ayache}, Ayache used wavelet method to show that the Hausdorff
dimension of the graph of the fractional Brownian sheet
$B_{\alpha/2}^{\#}$ is indeed equal to $d_{n,\alpha}$. This fact was
proved again by Ayache and Xiao \cite{Ayache_Xiao}  as a special
case of a more general result on Hausdorff dimension of fractional
Brownian sheets from $\R^n$ to $\R^d$, using a different method.
Therefore, it is natural for us to conjecture that the Hausdorff
dimension of the graph of GSGCC is also $d_{n,\alpha}$. In fact,
\eqref{eq1022_1} implies that
\begin{align*}
\sigma_{X^{\#}_{\alpha,\beta}}^2(\tau) =\left\langle
\left[X^{\#}_{\alpha,\beta}(t+\tau)-X^{\#}_{\alpha,\beta}(t)\right]^2\right\rangle=O\left(\Vert
\tau\Vert^{\min\alpha}\right)\hspace{1cm}\text{as}\;\;
\Vert\tau\Vert\rightarrow 0^+.
\end{align*} By a well-known
theorem (see e.g. \cite{Adler}), this implies that the Hausdorff
dimension of the graph of GSGCC is bounded above by $d_{n,\alpha}$.
On the other hand, it is easy to show that
\begin{lemma}\label{lemma1122_1}
Let $\mathcal{C}=\prod_{i=1}^{n}[a_i, b_i]$ be a hyperrectangle in
$\R^n$. There exist constants $c_1$ and $c_2$ such that

\begin{align}\label{eq1122_1}
c_1\sum_{i=1}^n |t_i-s_i|^{\alpha_i} \leq \left\langle
\left[X^{\#}_{\alpha,\beta}(t)-X^{\#}_{\alpha,\beta}(s)\right]^2\right\rangle
\leq c_2\sum_{i=1}^n |t_i-s_i|^{\alpha_i},
\end{align}for all $t,s\in \mathcal{C}$.
\end{lemma}Then by adapting the proof in \cite{Xiao} for the lower
bound on the Hausdorff dimension of general anisotropic Gaussian
fields with stationary increments, it can be verified that the graph
of GSGCC over a hyperrectangle has Hausdorff dimension equal to
$d_{n,\alpha}$.

    The condition for LRD can be
generalized to random sheet and it becomes \begin{align}\label{eq46}
\int\limits_{\R_+^n}\left|
C^{\#}_{\alpha,\beta}(\tau)\right|d^n\tau=\infty.
\end{align}
Since $$C^{\#}_{\alpha,\beta}(\tau)=\prod_{i=1}^n
C_{\alpha_i,\beta_i}(\tau_i),$$ and Proposition \ref{p2} for $n=1$
gives
\begin{align*}
\int\limits_{\R^+}\left|C_{\alpha_i,\beta_i}(\tau_i)\right|d\tau_i
=\infty \Longleftrightarrow 0<\alpha_i\beta_i\leq 1,
\end{align*}we have
\begin{proposition}
GSGCC is LRD if and only if for some $1\leq i\leq n$,
$0<\alpha_i\beta_i\leq 1$; it is SRD if and only if
$\alpha_i\beta_i>1$ for all $1\leq i\leq n$.
\end{proposition}

In order to determine the asymptotic behavior for the spectral
density of the generalized Cauchy sheet, we observe that the
spectral density of GSGCC, which we denote by
$S^{\#}_{\alpha,\beta}(\omega)$ can be expressed as products of
spectral densities of one--dimensional isotropic GFGCC considered in
section \ref{sec3}. Namely,
\begin{align}\label{eq1113_1}
S^{\#}_{\alpha,\beta}(\omega)=\prod_{i=1}^n
S_{\alpha_i,\beta_i}(\omega_i).
\end{align}
Therefore, the high frequency and low frequency behaviors of the
spectral density $S^{\#}_{\alpha,\beta}(\omega)$ can be obtained
from Proposition \ref{p1113_1} and Proposition \ref{p1113_2}
respectively.
\begin{proposition}
The high frequency limit of the spectral density
$S^{\#}_{\alpha,\beta}(\omega)$ is given by
\begin{align*}
S^{\#}_{\alpha,\beta}(\omega)\sim \prod_{i=1}^n
\mathcal{H}_i(\alpha_i,\beta_i;\omega_i),\hspace{1cm}\Vert\omega\Vert\rightarrow
\infty,
\end{align*}where
\begin{align*}
\mathcal{H}_i(\alpha_i,\beta_i;\omega_i)=\frac{ \beta_i}{\pi
}\Gamma\left(
 \alpha_i+1\right)\sin\frac{\pi\alpha_i}{2}|
\omega_i|^{-\alpha_i-1}, \hspace{1cm}\text{if}\;\; \alpha_i\in
(0,2);
\end{align*}and\begin{align*}
\mathcal{H}_i(\alpha_i,\beta_i;\omega_i)=
\frac{|\omega_i|^{\beta_i-1}}{2^{\beta_i}\Gamma(\beta_i)}e^{-|
\omega_i|},\hspace{2cm}\text{if}\;\;\alpha_i=2.\end{align*}
\end{proposition}

\begin{proposition}
The low frequency limit of the spectral density
$S^{\#}_{\alpha,\beta}(\omega)$ is given by
\begin{align*}
S^{\#}_{\alpha,\beta}(\omega)\sim \prod_{i=1}^n
\mathcal{L}_i(\alpha_i,\beta_i;\omega_i),\hspace{1cm}\Vert\omega\Vert\rightarrow
0^+,
\end{align*}where\begin{align*}
\mathcal{L}_i(\alpha_i,\beta_i;\omega_i)=&\frac{\Gamma(1-\alpha_i\beta_i)}{\pi}
\sin\frac{\pi\alpha_i
\beta_i}{2}|\omega|^{\alpha_i\beta_i-1},\hspace{1cm}\text{if}\;\;\alpha_i\beta_1<1;\\
\mathcal{L}_i(\alpha_i,\beta_i;\omega_i)=&\frac{1}{\pi}\left\{
\ln\frac{1}{|\omega_i|}-\beta_i\left(\psi(\beta_i)+\gamma\right)-\gamma\right\},\hspace{1cm}\text{if}\;\;\alpha_i\beta_i=1;\\
\mathcal{L}_i(\alpha_i,\beta_i;\omega_i)=&\frac{1}{\pi}\frac
{\Gamma\left(\frac{1}{\alpha_i}\right)\Gamma\left(\beta_i-\frac{1}{\alpha_i}\right)}{\alpha_i\Gamma(\beta_i)}
,\hspace{2.5cm}\text{if}\;\;\alpha_i\beta_i>1.
\end{align*}

\end{proposition}Here we have used the fact that
$\Gamma(2z)=2^{2z-1}\pi^{-\frac{1}{2}}\Gamma(z)\Gamma\left(z+\frac{1}{2}\right)$
(\cite{Gradshteyn_Ryzhik}, no.1 of 8.335),
$\Gamma(z)\Gamma(1-z)=\pi/\sin (\pi z)$ and $\psi(1/2)=-\gamma-2\ln
2$ (\cite{Gradshteyn_Ryzhik}, no.2 of 8.366).

\section{  Lamperti Transformation of GFGCC and GSGCC}   In his seminal paper \cite{Lamperti}, Lamperti introduced a
transformation which provides a one to one correspondence between a
self--similar process   and a stationary process. For a stationary
process $X(t)$, $t\in\R$, we let
\begin{align}\label{eq52}
Y(t)=t^HX(\ln t),
\end{align} for $t\in\R^+$, $H>0$, and $Y(0)=0$ be its $H$-Lamperti transfrom. Then
$Y(t)$ is an $H$-self--similar ($H$-ss) process. Conversely, if
$\{Y(t), t\geq 0\}$ is  $H$-ss with zero mean, then the inverse
Lamperti transformation of $Y(t)$ defined by
\begin{align}\label{eq53}
X(t)=e^{-Ht}Y(e^t), \hspace{0.5cm}t\in\R,
\end{align}
is a stationary process.

  There are two ways to define extensions of Lamperti transformation to $\R^n$ linking stationary random
field to a self-similar random field. Given a stationary random
field $X(t), t\in \R^n$, the first way to define its Lamperti
transformation is, given $H\in \R^+$, defined by
\begin{align}\label{eq1022_5}
Y(t)=\Vert t\Vert^{H}X(\ln t_1,\ldots, \ln t_n),
\end{align}for $t\in \R_+^n$. It is easy to show that
\begin{proposition}\label{p3}
Let $Y(t), t\in \R_+^n$ be the $H$-Lamperti transform of a
stationary  field $X(t), t\in\R^n$ defined by \eqref{eq1022_5}. Then
$Y(t)$ is a $H$-ss field.
\end{proposition}
\begin{proof}
For any $c\in \R^+$, we have
\begin{align*}
Y(ct)&=_d \Vert ct\Vert^H X\left(\ln [ct_1],\ldots, \ln
[ct_n]\right)\\
&=_dc^H\Vert t\Vert^H X(\ln t_1 +\ln c,\ldots, \ln t_n+\ln c)\\
&=_d c^H \Vert t\Vert^H X(\ln t_1,\ldots, \ln t_n)\\
&=_d c^H Y(t),
\end{align*}i.e. $Y(t)$ is $H$-ss.
\end{proof}
The inverse to this Lamperti transformation  transforms a random
field $Y(t)$ to
\begin{align}\label{eq1022_6}
X(t) =(e^{2t_1}+\ldots+e^{2t_n})^{-H/2} Y(e^{t_1}, \ldots, e^{t_n})
\end{align}for $t\in \R^n$. For a field $X(t),t\in\R^n$ to be
stationary, it must satisfy $X(t+u)=_d X(t)$ for all $u\in \R^n$.
This is an $n$-parameter family  of conditions. However, in proving
that $Y(t)$ is $H$-ss, we only use one parameter family of these
conditions, i.e., we only look at those $u$ of the form
$u_1=\ldots=u_n$. Therefore we cannot expect that the inverse
Lamperti transform \eqref{eq1022_6} of a $H$-ss field is a
stationary field. As an  example, consider the L$\acute{\text{e}}$vy
fractional Brownian field with index $H$, $B_{H}(t)$, which is a
$H$-ss field. Define its inverse $H$--Lamperti transform $Z(t)$ by
\eqref{eq1022_6}. We find that its covariance is
\begin{align*}
\Bigl\langle Z(t+\tau)Z(t)\Bigr\rangle
=&\frac{1}{2}\left(e^{2(t_1+\tau_1)}+\ldots
+e^{2(t_n+\tau_n)}\right)^{-H/2}\left(e^{2t_1 }+\ldots
+e^{2t_n}\right)^{-H/2}\times\\
&\Biggl\{ \left( e^{2(t_1+\tau_1)}+\ldots
+e^{2(t_n+\tau_n)}\right)^{H}+\left(e^{2t_1 }+\ldots
+e^{2t_n}\right)^{H}\\&-
\left([e^{t_1+\tau_1}-e^{t_1}]^2+\ldots+[e^{t_n+\tau_n}-e^{t_n}]^2\right)^{H}\Biggr\}.
\end{align*}It is easy to verify that for $n\geq 2$, this expression
is not independent of $t\in \R^n$. Therefore, $Z(t)$ is not a
stationary process.

To make full use of  the $n$-parameter families of symmetry in a
stationary process, there is another definition of Lamperti
transformation introduced by \cite{Genton_Perrin_Taqqu}. Given a
stationary process $X(t), t\in\R^n$, and a multi-index $H\in
\R_+^n$, the second $H$-Lamperti transform of $X(t)$, denoted by
$\mathbb{Y}(t)$, is defined as
\begin{align}\label{eq1022_8}
\mathbb{Y}(t)=t_1^{H_1}\ldots t_n^{H_n}X(\ln t_1,\ldots, \ln t_n)
\end{align}for $t\in \R_+^n$. It was shown in
\cite{Genton_Perrin_Taqqu} that
\begin{proposition}\label{p4}
Let $\mathbb{Y}(t), t\in \R_+^n$ be the $H$-Lamperti transform of a
stationary field $X(t), t\in\R^n$ defined by \eqref{eq1022_8}. Then
$\mathbb{Y}(t)$ is a $H$--mss field.
\end{proposition}The inverse of this second Lamperti transformation
transform a field $\mathbb{Y}(t), t\in\R_+^n$ to $X(t), t\in\R^n$,
where
\begin{align}\label{eq1022_9}
X(t)=e^{-\sum_{i=1}^n t_i H_i}\mathbb{Y}(e^{t_1},\ldots, e^{t_n}).
\end{align}
It has a nice property, namely, as was shown in
\cite{Genton_Perrin_Taqqu}:
\begin{proposition}
Let $X(t), t\in \R^n$ be the  inverse  second $H$--Lamperti
transform  of a $H$-mss field $\mathbb{Y}(t), t\in\R_+^n$ defined by
\eqref{eq1022_9}. Then $X(t)$ is a stationary field.
\end{proposition}

As a side remark, it can be shown that \eqref{eq1022_8} is
essentially the unique (up to some multiplicative constants)
transformation that takes a stationary process to a $H$-mss.
However, if for $c\in \R_+^n, \beta\in\R^+$, $f_{c,\beta}(t), t\in
\R_+^n$ is a function of the form
\begin{align*}
f_{c,\beta}(t)=\sum_{i=1}^n c_i t_i^{\beta},
\end{align*}then if $c[1],\ldots,c[m]\in \R_+^n$, $\beta_1,\ldots,\beta_m\in \R^+$ and
 $\gamma_1,\ldots,\gamma_m\in \R^+$ are such that $\beta_1\gamma_1+\ldots+\beta_m\gamma_m=H$, it is easy to verify that
the transform
$$X(t)\mapsto \prod_{i=1}^m f_{c[i],\beta_i}(t)^{\gamma_i}X(\ln t_1,\ldots, \ln t_n)$$ takes a stationary
process to a $H$--ss process, but its inverse in general does not
take a $H$--ss process to a stationary process. The first Lamperti
transformation \eqref{eq1022_5} corresponds to $m=1$,
$c[1]_1=\ldots=c[1]_n=1$, $\beta_1=2$ and $\gamma_1=H/2$.

Returning to the GFGCC $X_{\alpha,\beta}(t)$ and the GSGCC
$X^{\#}_{\alpha,\beta}(t)$, Proposition \ref{p3} shows that their
first Lamperti transforms, $Y_{\alpha,\beta}(t)$ and
$Y^{\#}_{\alpha,\beta} (t)$, are $H$--ss fields with covariances
\begin{align*}
\Bigl\langle Y_{\alpha,\beta}(t)Y_{\alpha,\beta}(s)\Bigr\rangle=&
\Vert t\Vert^{H}\Vert s\Vert^H \Bigl\langle X_{\alpha,\beta}(\ln
t_1,\ldots, \ln t_n)X_{\alpha,\beta}(\ln s_1,\ldots,\ln
s_n)\Bigr\rangle\\
=&\Vert t\Vert^{H}\Vert s\Vert^H\left[1+\left(\sum_{i=1}^n (\ln
t_i-\ln s_i)^2\right)^{\alpha/2}\right]^{-\beta},
\end{align*}and \begin{align*}
\Bigl\langle Y^{\#}_{\alpha,\beta} (t)Y^{\#}_{\alpha,\beta}
(s)\Bigr\rangle=&\Vert t\Vert^{H}\Vert s\Vert^H\prod_{i=1}^n
\left[1+ (\ln t_i-\ln s_i)^{\alpha_i}\right]^{-\beta_i}
\end{align*}respectively. On the other hand, Proposition  \ref{p4} shows
that their second Lamperti transforms,
$\mathbb{Y}_{\alpha,\beta}(t)$ and
$\mathbb{Y}^{\#}_{\alpha,\beta}(t)$, are $H$--mss fields with
covariances
\begin{align*}
\Bigl\langle
\mathbb{Y}_{\alpha,\beta}(t)\mathbb{Y}_{\alpha,\beta}(s)\Bigr\rangle
=&\prod_{i=1}^n t_i^{H_i}\prod_{i=1}^n
s_i^{H_i}\left[1+\left(\sum_{i=1}^n (\ln t_i-\ln
s_i)^2\right)^{\alpha/2}\right]^{-\beta},
\end{align*}and \begin{align*}
\Bigl\langle
\mathbb{Y}^{\#}_{\alpha,\beta}(t)\mathbb{Y}^{\#}_{\alpha,\beta}(s)\Bigr\rangle
=& \prod_{i=1}^n t_i^{H_i}s_i^{H_i}\left[1+ (\ln t_i-\ln
s_i)^{\alpha_i}\right]^{-\beta_i}
\end{align*}respectively.

Next we consider whether the LRD property of GFGCC and GSGCC is
preserved
  under the Lamperti transformations. First we note that the LRD condition \eqref{eq7} can be extended to
non-stationary field in the following way. \begin{definition}Let
\begin{align}\label{eq59}
R_X(t, t+\tau)=\frac{C(t,t+\tau)}{\sqrt{C(t+\tau,t+\tau)C(t,t)}}
\end{align} be the correlation function of a non-stationary Gaussian
field $X(t)$. Then the condition of LRD for the non-stationary
Gaussian field $X(t)$ is   given by \begin{align}\label{eq60}
\int_{\R_+^n} \left|R_X(t, t+\tau)\right|d^n\tau=\infty.
\end{align}
\end{definition} \begin{proposition}~\\
A.   The $H$-ss Gaussian field $Y_{\alpha,\beta}(t)$ is LRD for all $\alpha\in (0,2]$ and $\beta>0$.\\
B.  The $H$-ss Gaussian field $Y^{\#}_{\alpha,\beta} (t)$ is LRD for
all $\alpha, \beta \in \R_+^n$ satisfying $\alpha_i\in (0,2]$
and $\beta_i>0$, $1\leq i\leq n$.\\
C. The $H$-mss Gaussian field $\mathbb{Y}_{\alpha,\beta}(t)$ is LRD for all $\alpha\in (0,2]$ and $\beta>0$.\\
D.  The $H$-mss Gaussian field $\mathbb{Y}^{\#}_{\alpha,\beta}(t)$
is LRD for all $\alpha, \beta \in \R_+^n$ satisfying $\alpha_i\in
(0,2]$
and $\beta_i>0$, $1\leq i\leq n$.\\
\end{proposition}\begin{proof}It is easy to verify that   both the processes $Y_{\alpha,\beta}(t)$
and $\mathbb{Y}_{\alpha,\beta}(t)$ have the same correlation
function
\begin{align}\label{eq61}
&R_{Y_{\alpha,\beta}}(t+\tau,
t)=R_{\mathbb{Y}_{\alpha,\beta}}(t+\tau,
t)\\=&\left[1+\left(\sum_{i=1}^n\left(\ln
\left[1+\frac{\tau_i}{t_i}\right]\right)^2\right)^{\alpha/2}\right]^{-\beta},
\hspace{0.5cm}\tau\in \R_+^n,\nonumber
\end{align}which is independent of $H$. In order to show that these processes are
LRD,
 we have by condition \eqref{eq60},
\begin{align*}
\int\limits_{\R_+^n}\left|R_{Y_{\alpha,\beta}}(t+\tau,
t)\right|d^n\tau=&\int\limits_{\R_+^n}\left[1+\left(\sum_{i=1}^n\left(\ln
\left[1+\frac{\tau_i}{t_i}\right]\right)^2\right)^{\alpha/2}\right]^{-\beta}d^n\tau\\
=&\left[\prod_{i=1}^n
t_i\right]\int\limits_{\R_+^n}\left[1+\left(\sum_{i=1}^n\left(\ln
\left[1+u_i\right]\right)^2\right)^{\alpha/2}\right]^{-\beta}d^nu\\
=&\left[\prod_{i=1}^n
t_i\right]\int\limits_{\R_+^n}\left[1+\left(\sum_{i=1}^nv_i^2\right)^{\alpha/2}\right]^{-\beta}\left[\prod_{i=1}^n
e^{v_i}\right]d^nv,
\end{align*}
 where we have used the substitutions $u_i=\tau_i/t_i$ and $v_i=\ln(1+u_i)$. Clearly,
when $\Vert v\Vert \rightarrow \infty$, the integrand also
approaches $\infty$. Therefore,
 the last integral
diverges for all $\alpha,\beta$, which is the condition for
$Y_{\alpha,\beta}(t)$ and $\mathbb{Y}_{\alpha,\beta}(t)$ to be LRD.
The statements for $Y^{\#}_{\alpha,\beta}(t)$ and
$\mathbb{Y}^{\#}_{\alpha,\beta}(t)$ are proved
analogously.\end{proof}
    Note that the LRD property of GFGCC and GSGCC is preserved under
the Lamperti transformations, but the SRD property is not preserved.
 Here we have an example that the application of
Lamperti transformation to a LRD stationary process (in this case
the GFGCC and GSGCC) gives a (multi)--self--similar process with
LRD. Examples of Lamperti transformation encountered so far relate
either two short memory processes (for example, Ornstein-Uhlenbeck
process and Brownian motion), or between a SRD process and a LRD
process (in the case of fBm and its inverse Lamperti transformed
process). For examples, one can show that the inverse Lamperti
transformation of fractional Brownian sheet with LRD property gives
rise to a stationary random sheet with short range dependence, and
the stationary field associated with the L$\acute{\text{e}}$vy
fractional Brownian field is a stationary field with SRD.

     Recall  that L$\acute{\text{e}}$vy fractional Brownian field is
essentially the only self--similar Gaussian field with stationary
increments \cite{Samorodnitsky_Taqqu}. Similarly, one can show that
fractional Brownian sheet is essentially the only
multi--self--similar Gaussian random field that has stationary total
increments. Hence $Y_{\alpha,\beta}(t)$,
 $Y^{\#}_{\alpha,\beta}(t)$, $\mathbb{Y}_{\alpha,\beta}(t)$ and
$\mathbb{Y}^{\#}_{\alpha,\beta}(t)$, the self-similar and
multi--self--similar fields associated with GFGCC and GSGCC, do not
have stationary increments or total increments. There exists a
weaker stationary property known as asymptotically locally
stationarity, which requires the field to be stationary in the limit
$\Vert\tau\Vert\rightarrow 0^+$.

\begin{definition}
A centered Gaussian random field $X(t)$ is said to have
asymptotically locally stationary increment if and only if as
$\Vert\tau\Vert\rightarrow 0^+$, the variance of its increment
$\sigma^2_t(\tau)=\left\langle
\left[\Delta_{\tau}X(t)\right]^2\right\rangle$ is independent of
$t$. More precisely,
\begin{align}\label{eq1023_5}\sigma_t^2(\tau)=f(\tau)+ g(t,\tau),
\end{align}where $f(\tau)$ is independent of $t$, and for any fixed $t$,
$g(t,\tau)=o(f(\tau))$ as functions of $\tau$. Similarly, a centered
Gaussian random field $X(t)$ is said to have asymptotically locally
stationary total increment if and only if as
$\Vert\tau\Vert\rightarrow 0^+$, the variance of its total increment
$\hat{\sigma}^2_t(\tau)=\left\langle
\left[\square_{\tau}X(t)\right]^2\right\rangle$ is independent of
$t$.
\end{definition}

For the random field $Y_{\alpha,\beta}(t)$, we have
\begin{align*}
&\left\langle
\left[\Delta_{\tau}Y_{\alpha,\beta}(t)\right]^2\right\rangle
=\left\langle \left[Y_{\alpha,\beta}(t+\tau)\right]^2\right\rangle+
\left\langle
\left[Y_{\alpha,\beta}(t)\right]^2\right\rangle-2\Bigl\langle
 Y_{\alpha,\beta}(t+\tau)
 Y_{\alpha,\beta}(t)\Bigr\rangle\\
&=\Vert t+\tau\Vert^{2H} + \Vert t\Vert^{2H}-2\Vert
t+\tau\Vert^{H}\Vert t\Vert^H \left[1+\left(\sum_{i=1}^n \left(\ln
\left[1+\frac{\tau_i}{t_i}\right]\right)^2\right)^{\alpha/2}\right]^{-\beta}.
\end{align*}Using
\begin{align*}
\Vert t+\tau\Vert
=\sqrt{(t_1+\tau_1)^2+\ldots+(t_n+\tau_n)^2}=&\Vert t\Vert \left(1+
\frac{2\sum_{i=1}^n t_i\tau_i}{\Vert
t\Vert^2}+O(\Vert\tau\Vert^2)\right)^{1/2}\\
=&\Vert t\Vert \left(1+ \frac{\sum_{i=1}^n t_i\tau_i}{\Vert
t\Vert^2}+O(\Vert\tau\Vert^2)\right)
\end{align*}and
\begin{align*}
\left[1+\left(\sum_{i=1}^n \left(\ln
\left[1+\frac{\tau_i}{t_i}\right]\right)^2\right)^{\alpha/2}\right]^{-\beta}=&
\left[1+\left(\sum_{i=1}^n \left(
\frac{\tau_i}{t_i}+O(\tau_i^2)\right)^2\right)^{\alpha/2}\right]^{-\beta}\\
=&\left[1+\left(\sum_{i=1}^n \left[
\frac{\tau_i}{t_i}\right]^2+O(\Vert \tau\Vert^3)
\right)^{\alpha/2}\right]^{-\beta}\\
=&\left[1+\left(\sum_{i=1}^n \left[
\frac{\tau_i}{t_i}\right]^2\right)^{\alpha/2}+O(\Vert
\tau\Vert^{\alpha+1})
 \right]^{-\beta}\\
 =&1-\beta\left(\sum_{i=1}^n \left[
\frac{\tau_i}{t_i}\right]^2\right)^{\alpha/2}+O\left(\Vert\tau\Vert^{\min\{2\alpha,
\alpha+1\}}\right),
\end{align*}we find that as $\Vert \tau\Vert \rightarrow 0^+$,
\begin{align}\label{eq1026_1}
\left\langle
\left[\Delta_{\tau}Y_{\alpha,\beta}(t)\right]^2\right\rangle=&\Vert
t\Vert^{2H}\left(1+2H\frac{\sum_{i=1}^n t_i\tau_i}{\Vert
t\Vert^2}\right)+\Vert t\Vert^{2H} \\\nonumber &- 2\Vert
t\Vert^{2H}\left(1+H\frac{\sum_{i=1}^n t_i\tau_i}{\Vert
t\Vert^2}\right)\left(1-\beta\left(\sum_{i=1}^n \left[
\frac{\tau_i}{t_i}\right]^2\right)^{\alpha/2}\right)\\&+O\left(\Vert\tau\Vert^{\min\{2,2\alpha,
\alpha+1\}}\right)\nonumber\\ \nonumber=& 2\beta \Vert
t\Vert^{2H}\left(\sum_{i=1}^n \left[
\frac{\tau_i}{t_i}\right]^2\right)^{\alpha/2}+O\left(\Vert\tau\Vert^{\min\{2,2\alpha,
\alpha+1\}}\right).
\end{align}It is obvious that the leading term is not independent of
$t$ for any $H$ and $\alpha$, unless when $n=1$ and $\alpha=2H$.
Therefore for $n\geq 2$, $Y_{\alpha,\beta}(t)$ does not have
asymptotically locally stationary increments. Similarly,   for the
fields $\mathbb{Y}_{\alpha,\beta}(t)$, $Y^{\#}_{\alpha,\beta}(t)$
and $\mathbb{Y}^{\#}_{\alpha,\beta}(t)$, one can show similarly that
as $\Vert \tau\Vert \rightarrow 0^+$,
\begin{align}\label{eq1026_4}
\left\langle
\left[\Delta_{\tau}\mathbb{Y}_{\alpha,\beta}(t)\right]^2\right\rangle=&
 2\beta t_1^{2H_1}\ldots t_n^{2H_n}\left(\sum_{i=1}^n \left[
\frac{\tau_i}{t_i}\right]^2\right)^{\alpha/2}+O\left(\Vert\tau\Vert^{\min\{2,2\alpha,
\alpha+1\}}\right),
\end{align}
\begin{align}\label{eq1026_2}
\left\langle \left[\Delta_{\tau}Y^{\#}_{\alpha,\beta}
(t)\right]^2\right\rangle=& 2\Vert t\Vert^{2H}\sum_{i=1}^n \beta_i
\left|\frac{\tau_i}{t_i}\right|^{\alpha_i}+O\left(\Vert\tau\Vert^{\min\{2,2\alpha_i,
\alpha_i+1\}}\right),
\end{align}\begin{align}\label{eq1026_3}
\left\langle \left[\Delta_{\tau}\mathbb{Y}^{\#}_{\alpha,\beta}
(t)\right]^2\right\rangle=& 2t_1^{2H_1}\ldots t_n^{2H_n}\sum_{i=1}^n
\beta_i
\left|\frac{\tau_i}{t_i}\right|^{\alpha_i}+O\left(\Vert\tau\Vert^{\min\{2,2\alpha_i,
\alpha_i+1\}}\right).
\end{align}Therefore for $n\geq2$, none of the fields
$\mathbb{Y}_{\alpha,\beta}(t)$, $Y^{\#}_{\alpha,\beta}(t)$ and
$\mathbb{Y}^{\#}_{\alpha,\beta}(t)$  have asymptotically locally
stationary increments. When $n=1$, all the fields
$Y_{\alpha,\beta}(t)$, $\mathbb{Y}_{\alpha,\beta}(t)$,
$Y^{\#}_{\alpha,\beta}(t)$ and $\mathbb{Y}^{\#}_{\alpha,\beta}(t)$
are actually the same, and they have asymptotically locally
stationary increment if and only if $\alpha=2H$, in which case the
variance of the increment $\sigma_t^2(\tau)$ behaves like
$$\sigma_t^2(\tau)\sim 2\beta|\tau|^{\alpha}\hspace{1cm}|\tau|\rightarrow 0.$$  Next we consider the total increments.
Since increment is the same as total increment when  $n=1$, we only
need to consider $n\geq 2$. Using similar computations as given
above, one can verify that for $n\geq 2$, the fields
$Y_{\alpha,\beta}(t)$, $\mathbb{Y}_{\alpha,\beta}(t)$ and
$Y^{\#}_{\alpha,\beta}(t)$ do not have asymptotically locally
stationary increments, but $\mathbb{Y}^{\#}_{\alpha,\beta}(t)$ have
if $\alpha=2H$. We show the computation of the latter case here. By
definition,
\begin{align*}
&\left\langle
\left[\square_{\tau}\mathbb{Y}^{\#}_{\alpha,\beta}(t)\right]^2\right\rangle=\sum_{\delta\in\{0,1\}^n}
\sum_{\eta\in\{0,1\}^n} (-1)^{\sum_{i=1}^n \delta_i+\sum_{i=1}^n
\eta_i}\\&\hspace{4cm}\left\langle
\mathbb{Y}^{\#}_{\alpha,\beta}\left(t+\sum_{i=1}^n\delta_i\tau_i
e_i\right)\mathbb{Y}^{\#}_{\alpha,\beta}\left(t+\sum_{i=1}^n\eta_i\tau_i
e_i\right) \right\rangle\\
=&\prod_{i=1}^n
\sum_{(\delta_i,\eta_i)\in\{0,1\}^2}(-1)^{\delta_i+\eta_i}(t_i+\delta_i\tau_i)^{H_i}
(t_i+\eta_i\tau_i)^{H_i}C_{\alpha_i,\beta_i}\Bigl(\ln[t_i+\delta_i\tau_i],
\ln[t_i+\eta_i\tau_i]\Bigr)\\
=&\prod_{i=1}^n \Biggl\{t_i^{2H_i}-2(t_i+\tau_i)^{H_i}t_i^{H_i}
\left(1+\left|\ln\left[1+\frac{\tau_i}{t_i}\right]\right|^{\alpha_i}\right)^{-\beta_i}+(t_i+\tau_i)^{2H_i}\Biggr\}\\
=&\prod_{i=1}^n\left\{2\beta_i
t_i^{2H_i}\left|\frac{\tau_i}{t_i}\right|^{\alpha_i}+O\left(|\tau_i|^{\min\{2,2\alpha_i,\alpha_i+1\}}\right)\right\}\\
=&2^n\prod_{i=1}^n\left\{\beta_i
t_i^{2H_i}\left|\frac{\tau_i}{t_i}\right|^{\alpha_i}\right\}+O\left(\Vert
\tau\Vert^{\sum_{i=1}^n\alpha_i+\delta}\right),
\end{align*}where $\delta=\min\{ 2-\alpha_i, \alpha_i, 1\}_{i=1}^n$.
If $\alpha=2H$, then as $\Vert\tau\Vert\rightarrow 0^+$,
\begin{align*}
\left\langle
\left[\square_{\tau}\mathbb{Y}^{\#}_{\alpha,\beta}(t)\right]^2\right\rangle\sim
2^n\prod_{i=1}^n [\beta_i|\tau_i|^{\alpha_i}].
\end{align*} Therefore, we have verified the following:
\begin{proposition}
The total increments of $\mathbb{Y}^{\#}_{\alpha,\beta}(t)$   are
asymptotically locally stationary if $\alpha=2H$.
\end{proposition}

Next we consider the tangent fields (Definition \ref{def1}) of the
Lamperti transforms of GFGCC $Y_{\alpha,\beta}(t)$ and
$\mathbb{Y}_{\alpha,\beta}(t)$. We have
\begin{proposition}\label{p1026_1}~\\
A. The field $Y_{\alpha,\beta}(t)$ is lass of order $\alpha/2$, with
tangent field at   $t\in\R_+^n$ being
$$\sqrt{2\beta}\Vert t\Vert^H
B_{\alpha/2}\left(\frac{u_1}{t_1},\ldots,\frac{u_n}{t_n}\right).$$\\
B. The field $\mathbb{Y}_{\alpha,\beta}(t)$ is lass of order
$\alpha/2$, with tangent field at   $t\in\R_+^n$ being
$$\sqrt{2\beta}\left[\prod_{i=1}^n t_i^{H_i}\right]
B_{\alpha/2}\left(\frac{u_1}{t_1},\ldots,\frac{u_n}{t_n}\right).$$\\

\end{proposition}
\begin{proof}
Using the formula
\begin{align*}
&\left\langle\Delta_{\vep u}Y_{\alpha,\beta}(t)\Delta_{\vep
v}Y_{\alpha,\beta}(t)
\right\rangle\\
=&\frac{1}{2}\Bigl\{ \left\langle\left[\Delta_{\vep
u}Y_{\alpha,\beta}(t)\right]^2\right\rangle
+\left\langle\left[\Delta_{\vep
v}Y_{\alpha,\beta}(t)\right]^2\right\rangle-\left\langle\left[\Delta_{\vep
(u-v)}Y_{\alpha,\beta}(t+\vep v)\right]^2\right\rangle\Bigr\},
\end{align*}we obtain immediately from \eqref{eq1026_1} that as
$\vep\rightarrow 0$,
\begin{align*}
&\left\langle \frac{\Delta_{\vep
u}Y_{\alpha,\beta}(t)}{\vep^{\alpha/2}}\frac{\Delta_{\vep
v}Y_{\alpha,\beta}(t)}{\vep^{\alpha/2}} \right\rangle\\
\sim &\beta\Vert t\Vert^{2H}\left\{
\left(\sum_{i=1}^n\left[\frac{u_i}{t_i}\right]^2\right)^{\alpha/2}+
\left(\sum_{i=1}^n\left[\frac{v_i}{t_i}\right]^2\right)^{\alpha/2}-
\left(\sum_{i=1}^n\left[\frac{u_i-v_i}{t_i}\right]^2\right)^{\alpha/2}\right\}.
\end{align*}Notice that this is up to the factor $2\beta \Vert
t\Vert^{2H}$, the covariance of the scaled L$\acute{\text{e}}$vy
Brownian field
$$\left\{B_{\alpha/2}\left(\frac{u_1}{t_1},\ldots,\frac{u_n}{t_n}\right)\,:\,
u\in\R^n\right\}.$$The statement for $\mathbb{Y}_{\alpha,\beta}(t)$
is proved similarly.
\end{proof}
In fact, one can also deduce from \eqref{eq1026_1} and
\eqref{eq1026_4} that the local fractal index of the fields
$Y_{\alpha,\beta}(t)$ and $\mathbb{Y}_{\alpha,\beta}(t)$ are both
equal to $\alpha/2$, and the Hausdorff dimension of their graphs
over a hyperrectangle is $n+1-\alpha/2$.

For the tangent fields of the Lamperti transforms of  GSGCC
$Y^{\#}_{\alpha,\beta}(t)$ and $\mathbb{Y}^{\#}_{\alpha,\beta}(t)$,
we recall that given $\alpha\in (0,2]^n$, $\min \alpha=\min
\{\alpha_i\}_{i=1}^n$, and we assume WLOG that
$\min\alpha=\alpha_1=\dots=\alpha_{m_{\alpha}}<\alpha_{m_{\alpha}+1}\leq
\ldots\leq\alpha_n$ for some $1\leq m_{\alpha}\leq n$. Then as in
Proposition \ref{p1026_1}, we can use \eqref{eq1026_2} and
\eqref{eq1026_3} to show that
\begin{proposition}\label{p1026_3}
~\\
A. The field $Y^{\#}_{\alpha,\beta}(t)$ is lass of order $\min
\alpha/2$, with tangent field at   $t\in\R_+^n$ being $$ \Vert
t\Vert^H
T_{\alpha,\beta}\left(\frac{u_1}{t_1},\ldots,\frac{u_n}{t_n}\right),$$where
the field $T_{\alpha,\beta}(u), u \in \R^n$ is defined
in Proposition \ref{p1026_2}.\\
B. The field $\mathbb{Y}^{\#}_{\alpha,\beta}(t)$ is lass of order
$\min\alpha/2$, with tangent field at   $t\in\R_+^n$ being $$
\left[\prod_{i=1}^n t_i^{H_i}\right]
T_{\alpha,\beta}\left(\frac{u_1}{t_1},\ldots,\frac{u_n}{t_n}\right).$$\\

\end{proposition}
From Propositions \ref{p1} and \ref{p1026_1} and Propositions
\ref{p1026_2} and \ref{p1026_3}, we find that the tangent fields of
GFGCC and GSGCC are related to the tangent fields of their Lamperti
transforms by  some change of variable formulas. Namely, if
  $X(t)$ has tangent field $T(u)$ at $t\in \R^n$, then its first
  Lamperti transform $Y(s)$ has tangent field $$\Vert
  s\Vert^H T\left(\frac{u_1}{s_1},\ldots,\frac{u_n}{s_n}\right)$$
  at $s\in\R_+^n$, and its second Lamperti transform
  $\mathbb{Y}(s)$ has tangent field $$\left[\prod_{i=1}^n s_i^{H_i}\right]
  T\left(\frac{u_1}{s_1},\ldots,\frac{u_n}{s_n}\right)$$ at
  $s\in\R_+^n$. We also notice   that the order of
self-similarity of the fields $Y_{\alpha,\beta}(t)$ and
$Y^{\#}_{\alpha,\beta}(t)$, $H$, is in general different from their
order of local asymptotic self--similarity, being $\alpha/2$ and
$\min\alpha/2$ respectively.

Since the stochastic process $\mathbb{Y}^{\#}_{\alpha,\beta}(t)$
resembles the GSGCC $X^{\#}_{\alpha,\beta}(t)$ in the sense that
both their covariances can be written as products of covariance of
one-dimensional processes, we can consider the limit of the total
increments of $\mathbb{Y}^{\#}_{\alpha,\beta}(t)$. It is easy to
obtain as in Proposition \ref{p1026_5}, using the $n=1$ case of
Proposition \ref{p1026_1} that
\begin{proposition}
\begin{align*}
\lim_{\vep \rightarrow 0^+}\left\langle \frac{\square_{\vep.u}
\mathbb{Y}^{\#}_{\alpha,\beta}(t)}{\prod_{i=1}^n \vep_i^{\alpha_i/2}
}\right\rangle_{u\in\R^n}=_d \left[\prod_{i=1}^n \sqrt{2\beta_i
}t_i^{H_i-(\alpha_i/2)}\right]B_{\alpha/2}^{\#}(u).
\end{align*}Here $\vep.u=\sum_{i=1}^n \vep_i u_i e_i$.
\end{proposition}

\section{ Concluding Remarks}
We have studied some of the basic properties of GFGCC and GSGCC, and
their associated Lamperti transforms. The asymptotic properties of
the spectral densities of GFGCC and GSGCC are considered. One
expects the separate characterization of fractal dimension and long
range dependence for GFGCC and GSGCC  will provide more flexibility
in their applications to modeling various surfaces and images. In
the one-dimensional case, GFGCC has been applied to model the
Havriliak-Negami relaxation law \cite{Lim_Li}. The estimations of
parameters for stationary Gaussian processes have been widely
studied. Some of these estimations can be adapted for GFGCC and
GSGCC, and they are crucial to the applications. Further
applications in spatial-temporal processes are possible if GFGCC is
extended and modified to a space-time field to include
non-stationarity  and anisotropy \cite{i, ii, iii, c, e, f, Dimri}.
We hope to apply results obtained in this paper to model various
physical systems such as thin film surfaces in semiconductors
\cite{Palasantzas_Hosson, Sahoo_Thakur_Tokas}, surface ocean waves
\cite{Berizzi_Mese_Martorella}, geological morphology \cite{Dimri},
etc, in a future work.

\vspace{1cm}\textbf{Acknowledgements} The authors would like to
thank Malaysian Academy of Sciences, Ministry of Science, Technology
and Innovation for funding this project under the Scientific
Advancement Fund Allocation (SAGA) Ref. No P96c.

\end{document}